\documentclass[11pt]{amsart}
\usepackage[utf8]{inputenc} 
\usepackage{fix-cm}   
\usepackage[T1]{fontenc}

\usepackage{lmodern}  
\usepackage[margin=1in,a4paper]{geometry} 
\usepackage{amsmath,amsthm,amssymb} 
\numberwithin{equation}{section}
\usepackage{bm} 
\usepackage[dvipsnames]{xcolor}  
\usepackage{microtype}
\usepackage{indentfirst}
\usepackage{mathtools} 
\usepackage{accents} 
\usepackage{adjustbox} 
\usepackage[most]{tcolorbox} 
\usepackage{tikz,tkz-euclide} 
\usepackage{tikz-cd} 
\usepackage{pgf} 
\usepackage{stmaryrd} 
\usepackage{dsfont} 
\usepackage{float}
\usepackage{pdfpages}
\usepackage{textgreek}
\usepackage{mathrsfs} 
\usepackage{lscape}
\usepackage{array}
\usepackage{comment}
\usepackage{thmtools}
\usepackage{thm-restate}
\usepackage{subcaption}
\usepackage{wrapfig}
\usepackage{xfrac}
\usepackage{nicefrac}
\usepackage[font=small,labelfont=bf]{caption}
\usetikzlibrary{fadings}
\usetikzlibrary{arrows.meta}
\usetikzlibrary{chains,positioning,scopes,graphs, graphs.standard,intersections,calc}

\usepackage{latexsym,amsmath,amsthm,amsfonts,amscd,amssymb,mathdots,mathtools}
\usepackage[utf8]{inputenc}
\usepackage{stmaryrd}
\usepackage[mathscr]{eucal}
\usepackage{mathrsfs}
\usepackage{enumerate}
\usepackage[margin=1in]{geometry}
\usepackage{mathrsfs}
\usepackage{amssymb}
\usepackage[all]{xy}
\usepackage{xcolor}
\usepackage{graphics}
\usepackage{lscape}
\usepackage{array}
\usepackage{xy}

\newcommand{\N}{\mathbb{N}}
\newcommand{\Z}{\mathbb{Z}}
\newcommand{\C}{\mathbb{C}}
\newcommand{\R}{\mathbb{R}}

\newcommand{\sll}{\mathfrak{s}\mathfrak{l}}

\newcommand{\Hom}{\operatorname{Hom}}
\newcommand{\End}{\operatorname{End}}

\newcommand{\rk}{\operatorname{rk}}
\newcommand{\ev}{\operatorname{ev}}
\newcommand{\codim}{\operatorname{codim}}
\newcommand{\gl}{\mathfrak{gl}}
\renewcommand{\phi}{\varphi}
\newcommand{\Jac}{\operatorname{Jac}}

\newcommand{\Prym}{\operatorname{Prym}}
\newcommand{\Nm}{\operatorname{Nm}}
\newcommand{\Pic}{\operatorname{Pic}}
\newcommand{\PGL}{\operatorname{PGL}}
\newcommand{\PSL}{\operatorname{PSL}}
\newcommand{\GL}{\operatorname{GL}}
\newcommand{\SL}{\operatorname{SL}}
\newcommand{\SU}{\operatorname{SU}}
\newcommand{\im}{\operatorname{Im}}

\newcommand{\rank}{\operatorname{rank}}
\newcommand{\tr}{\operatorname{tr}}
\newcommand{\id}{\operatorname{id}}

\newcommand{\Ndet}{\operatorname{Ndet}}

\renewcommand{\div}{\operatorname{div}}

\newcommand{\Aut}{\operatorname{Aut}}
\newcommand{\coim}{\operatorname{coim}}
\newcommand{\dolb}{\bar\partial}

\newcommand{\Lp}{\mathopen{\raisebox{-0.4em}{$\Biggl($}}}
\newcommand{\Rp}{\mathclose{\raisebox{-0.4em}{$\Biggr)$}}}

\newcommand{\fullquot}[2]{%
	#1\mkern-2mu
	\raisebox{-0.25ex}{\rotatebox{-18}{/}}
	\mkern+2mu
	\raisebox{-0.65ex}{$#2$}
}

\newcommand{\why}[1][check]{\textcolor{BrickRed}{\footnotesize \textbf{[#1]}}}

\newcommand*\circled[1]{\tikz[baseline=(char.base)]{
            \node[shape=circle,draw,inner sep=2pt] (char) {#1};}}

\newtheorem{theorem}{Theorem}[section]
\newtheorem{lemma}[theorem]{Lemma}
\newtheorem{cor}[theorem]{Corollary}
\newtheorem{prop}[theorem]{Proposition}
\newtheorem*{conj*}{\textcolor{red}{Conjecture}}

\theoremstyle{definition}
\newtheorem{deff}[theorem]{Definition}
\newtheorem*{deff*}{Definition}
\newtheorem{example}[theorem]{Example}

\theoremstyle{remark}
\newtheorem{rmk}[theorem]{Remark}

\usepackage[backend=biber,giveninits=true,style=alphabetic,maxbibnames=99,maxalphanames=4]{biblatex}
\usepackage{fancyhdr}
\usepackage[breaklinks]{hyperref} 
\usepackage{bookmark}

\newcommand{\ThmA}{
  Let $F_\alpha$ be a $\C^*$-fixed point component inside the moduli space $\mathcal{M}$. Then the subvariety $\mathcal L_\alpha$ is a $\C^*$-invariant Lagrangian affine fibration over $F_\alpha$ such that, for every $\mathcal E \in F_\alpha$, we have 
  \[T_{\mathcal E}\mathcal L_\alpha = T^0_{\mathcal E} \mathcal{M} \oplus \bigoplus_{j\ge 2} T^j_{\mathcal E}\mathcal M.\]
}

\newcommand{\ThmB}{
	Let $F_\alpha$ be a component of type $(1,\ldots,1)$ fixed points with $\deg \Phi_i = m_i$ and let $a \in \mathcal A$ be in the conditions of Proposition \ref{prop:generic}.
	Then, for $n \ge 3$ odd,
\[|\mathcal L_\alpha \cap h^{-1}(a)| =  n^{2g}\prod_{i=1}^{\frac{n-1}{2}}\binom{\binom{n}{i}(2g-2)}{m_i,m_{n-i}}   \] and for $n \ge 4$ even,
\[|\mathcal L_\alpha \cap h^{-1}(a)| = n^{2g}\binom{\binom{n}{n/2}(g-1)}{m_{n/2}}2^{m_{n/2}}\prod_{i=1}^{\frac{n}{2}-1}\binom{\binom{n}{i}(2g-2)}{m_i,m_{n-i}}.\]
}

\newcommand{\ThmC}{
	A fixed point $\mathcal E$ of type $(1,\ldots,1)$ is $Z$-very stable if and only if the only points of the associated divisor $\delta_{+} = \delta_1 + \cdots + \delta_{n-1}$ with multiplicity bigger than 1 (if any) are concentrated in precisely one of the extremities, $\delta_1$ or $\delta_{n-1}$.
}

\newcommand{\ThmD}{
	The relative Fourier-Mukai transform $S$ over $\mathcal A^0$ satisfies
\[S\left(\left.\mathcal O_{\mathcal L_\alpha}\right|_{\mathcal M^0}\right) =  \Lambda^0_\alpha|_{\mathcal M_{\PGL}^0}.\] 
}

\author[P.~B.\ Gothen]{Peter B.\ Gothen}
\address{P.~B. Gothen,
	\newline\indent Centro de Matemática da Universidade do Porto,
	\newline\indent Departamento de Matemática, Faculdade de Ciências, Universidade do Porto,
	\newline\indent Rua do Campo Alegre s/n, 4169-007 Porto, Portugal}
\email{pbgothen@fc.up.pt}

\author[A. Oliveira]{André Oliveira}
\address{A. Oliveira,
	\newline\indent Centro de Matemática da Universidade do Porto,
	\newline\indent Departamento de Matemática, Faculdade de Ciências, Universidade do Porto,
	\newline\indent Rua do Campo Alegre s/n, 4169-007 Porto, Portugal}
\email{andre.oliveira@fc.up.pt}

\author[R. Pereira]{Rodrigo Pereira}
\address{R. Pereira,
	\newline\indent Centro de Matemática da Universidade do Porto,
	\newline\indent Departamento de Matemática, Faculdade de Ciências, Universidade do Porto,
	\newline\indent Rua do Campo Alegre s/n, 4169-007 Porto, Portugal}
\email{rodrigo.pereira@fc.up.pt}

\title{Zero velocity Lagrangians in the Hitchin system}

\thanks{The authors were partially supported by FCT (Fundação para a Ciência e Tecnologia), under the projects with reference UID/00144/2025, and associated DOI
	https://doi.org/10.54499/UID/00144/2025, CMUP, member of LASI, and 2024.15931.PEX \textit{Higgs bundles: geometry, algebra and physics} with associated DOI https://doi.org/10.54499/2024.15931.PEX.
	This research was supported in part by the International Centre for Theoretical Sciences (ICTS) for the program - Geometric Structures and Stability (code: ICTS/GSS2026/02).
	The third author was supported by the FCT, under the PhD scholarship with reference 2023.00922.BD. The third author also wishes to thank Brian Collier and Miguel González for enlightening discussions.}

\begin{document}

	\begin{abstract}
		Lagrangians of the moduli space of Higgs bundles are a fundamental piece of its study under the lens of mirror symmetry. We construct a new class of Lagrangians defined via a ``zero velocity condition'' imposed on the natural action of the non-zero complex numbers on the moduli space, motivated by a certain ``infinitesimal conormal principle''. These are Lagrangian subvarieties analogous to the upward flows considered by \cite{COLLIER20191193} and \cite{HH}, but which fiber over the $\C^*$-fixed point subvarieties. We then prove a formula for the number of intersection points of this Lagrangian with the generic fiber of the Hitchin map, and build a vector bundle over the dual moduli space which we conjecture to correspond to its mirror hyperholomorphic bundle.
	\end{abstract}
	
	\maketitle
	
	\tableofcontents \newpage

\section{Introduction}

The moduli space $\mathcal M$ of Higgs bundles $(E,\Phi)$ over a Riemann surface $X$ is a hyperkähler algebraic variety whose very rich geometry has been under intensive study since its introduction in \cite{hitchinSDERS}. It is an important object for algebraic and symplectic geometry, representation theory and theoretical physics (namely \emph{mirror symmetry} -- the symmetry under which two different Calabi-Yau manifolds define the ``same string theory'').

Of fundamental importance is the Hitchin map \cite{hitchinSDERS,hitchinSBIS} $h:\mathcal M \to \mathcal A$, with $\mathcal A$ a vector space (the Hitchin base), which essentially sends a pair $(E,\Phi)$ to the coefficients of the characteristic polynomial of $\Phi$. This map turns $\mathcal M$ into a completely integrable system \cite{BNR,hitchinSDERS,hitchinSBIS, bottacinSGMSSP, markmanSCIS} -- the celebrated Hitchin system -- and plays a central role in mirror symmetry \cite{donagi_pantev,FGOP1,HH,HT,KW,SYZ}.

The moduli space $\mathcal M$, being hyperkähler, possesses three complex structures $I,J,K$ satisfying the quaternion relations, together with three compatible real symplectic forms $\omega_I,\omega_J,\omega_K$. Following the terminology of \cite{KW}, a B-brane is roughly speaking a holomorphic bundle over a complex submanifold and an A-brane is a flat vector bundle over a Lagrangian submanifold. When three letters are grouped together, they refer to each of the three complex structures and symplectic forms of $\mathcal{M}$. Thus a BAA-brane is a holomorphic submanifold (with respect to $I$) which is Lagrangian with respect to the holomorphic symplectic form $\Omega_I = \omega_J + i\omega_K$, supporting a flat vector bundle, while a BBB-brane is a hyperholomorphic bundle (a bundle which is holomorphic with respect to the three complex structures) over a hyperkähler submanifold.

Mirror symmetry is predicted to exchange a BAA--brane of $\mathcal{M}$ with a BBB--brane of its mirror $\mathcal M^\vee$ \cite{KW}. 
This is because the Hitchin system $h: \mathcal{M} \to \mathcal A$ fits into the framework of the Strominger-Yau-Zaslow (SYZ) approach of obtaining mirror symmetric Calabi-Yau manifolds by replacing each nonsingular torus fiber by its dual and then attempting to extend it over the singular locus in the base $\mathcal A$. This justifies the relevance of studying the complex Lagrangian subvarieties (the supports of BAA-branes) of $\mathcal{M}$.

The starting point of this work is the observation that ``conormals give rise to Lagrangians''. By looking at the conormal spaces of a submanifold $S$ sitting inside a Lagrangian $L$, the idea is to find some other Lagrangian which agrees infinitesimally with the conormal bundle at the points of $S$.

We will apply this construction by considering the standard action of the non-zero complex numbers on the moduli space by scaling the Higgs field. Each connected component $F_\alpha$ of its fixed points  sits inside an irreducible component of the nilpotent cone -- the fiber over zero of the Hitchin map -- which is known to be Lagrangian.  At each fixed point $\mathcal E$, the tangent space decomposes into a sum of weighted spaces for the $\C^*$-action $T_{\mathcal E}^j \mathcal M$, where $j$ is an integer. The problem can then be reduced to that of the existence of a Lagrangian subvariety whose tangent space at each fixed point contains only weights $0$ or strictly bigger than $1$. Using the slice machinery of \cite{COLLIER20191193}, we show that indeed such an object, which we call $\mathcal L_\alpha$, exists. Moreover, it can be understood as those points of the moduli space whose $\C^*$-limit to $0$ lands in $F_\alpha$ with "zero velocity". More precisely, we extend the $\C^*$ orbit of a point $(E,\Phi)$ to a $\C$-orbit by adding its limit to $0$. The Lagrangian $\mathcal L_\alpha$ will then consist of those points whose derivative at $0$ vanishes.

\medskip
\noindent \textbf{Theorem \ref{thm:A}.} \textit{\ThmA}
\medskip

After building the Lagrangian $\mathcal L_\alpha$, we turn our attention to its interaction with the nilpotent cone. Based on \cite{HH}, we consider the "zero velocity upward flow" of a fixed point $\mathcal E$ -- the points that arrive with "zero velocity" to $\mathcal E$.
We characterize the fixed points of type $(1,\ldots,1)$ -- those that are direct sums of line bundles -- whose zero velocity upward flow contains nilpotent Higgs bundles (other than the fixed point itself) -- we call these $Z$-very stable, again based on the terminology of \cite{HH}. Since fixed points of type $(1,\ldots,1)$ are given by direct sums of line bundles $L_1 \oplus \cdots \oplus L_n$ over $X$ and maps $\Phi_i: L_i \to L_{i+1}K$ obtained by restricting the Higgs field, we obtain the following characterization by setting $\delta_i \coloneqq \div \Phi_i$.


\medskip
\noindent \textbf{Theorem \ref{thm:C}.} \textit{\ThmC}
\medskip

The proof is based on the observation that, while any non-reduced point of $\delta_+$ allows us to produce points with nilpotent Higgs field on the upward flow of the fixed point, when imposing the additional condition that the points must lie in the zero velocity locus, our choice of nonreduced points in $\delta_+$ is restricted, i.e., no "zero velocity nilpotents" can be produced from a multiple zero of $\delta_1$ or $\delta_{n-1}$.

We will then be interested in studying this Lagrangian under the lens of mirror symmetry, following the program of \cite{HH} where a dual BBB-brane was constructed for the upward flow of a very stable Higgs bundle of type $(1,\ldots,1)$. As a crucial first step towards that goal, we show that the intersection number of $\mathcal L_\alpha$ with a generic fiber of $h$ is finite and compute it. In order to do so, we first notice that, fixing a point $a$ in the Hitchin base, if $(E,\Phi)$ is in the fiber over $a$ and flows to a fixed point $\mathcal E$ with zero velocity, then there exist sections $R_{n,i}(a)$ of a power of $K$ such that $\div \Phi_i$ is a subdivisor of $\div R_{n,i}$. Moreover, we show in Proposition \ref{prop:generic} that there exists an open dense set $\mathcal A^0$ of the Hitchin base where the intersection points between the Lagrangian and a fiber over $\mathcal A^0$ all flow to very stable fixed points. This allows us to use a result of \cite{HH} to prove the following statement. 

\medskip
\noindent \textbf{Theorem \ref{thm:B}.} \textit{\ThmB}
\medskip

 Finally, we propose a description of what the mirror BBB brane should look like, again following the procedures of \cite{HH}. There, to each very stable fixed point $\mathcal E$ of type $(1,\ldots,1)$,  a vector bundle $\Lambda_{\mathcal E}$ over the moduli space is constructed using the twisted universal bundle. In our case, we find that, when computing the Fourier-Mukai transform of the intersection between the Lagrangian and a fiber of the Hitchin map over a point in $\mathcal A^0$, we obtain a direct sum, indexed by a finite set of fixed points $\mathcal E$, of subbundles of $\Lambda_{\mathcal{E}}$ restricted to the Hitchin fiber.
 
 We then build a vector bundle $\Lambda^0_\alpha$ over the open dense set $\mathcal M^0 = h^{-1}(\mathcal A^0)$ of the moduli space that agrees with the aforementioned Fourier-Mukai transform over each fiber $h^{-1}(a)$. Its rank will be given precisely by the numbers found in Theorem \ref{thm:B}. We also mention that, unlike \cite{HH}, which work over the moduli space of Higgs bundles for the group $\GL(n,\C)$, we consider the space of $\SL(n,\C)$-Higgs bundles. In particular this means that the vector bundle we construct will actually live over the moduli space of $\PGL(n,\C)$-Higgs bundles (since these groups are Langlands dual to each other). Thus, we have the following result.
 
 \medskip
 \noindent \textbf{Theorem \ref{thm:D}.} \textit{\ThmD}
 
 This paper is organized as follows. In section \ref{prelim}, we begin by reminding the reader the classic definitions and facts about $\SL(n,\C)$-Higgs bundles, their moduli space, the $\C^*$-action and the Hitchin system. We then proceed to the construction of our main object in section \ref{construct}, the "zero velocity Lagrangian" $\mathcal L_\alpha$ associated to a fixed point component $F_\alpha$ of the $\C^*$-action. In the remainder of the text, we will be exclusively working with type $(1,\ldots,1)$ fixed points, so in section \ref{main} we introduce their characterization in termos of divisors on the base curve $X$ and our main technical tool -- Hecke transformations of Higgs bundles. We also work out in detail some properties of the rank $2$ case, including the identification with the moduli space for $\SL(2,\R)$-Higgs bundles. We finish the section with a classification of type $(1,\ldots,1)$ into $Z$-very stable and $Z$-wobbly, according to the existence of nilpotent elements in their "zero velocity" upward flows. In section \ref{int_generic} we study the intersection between the Lagrangian and a generic fiber of the Hitchin map, we show that it is finite and count the number of points. Finally in section \ref{mirror} we build the vector bundle $\Lambda^0_\alpha$ over the dual moduli space of $\PGL(n,\C)$-Higgs bundles, which agrees generically with the Fourier-Mukai transform of the structure sheaf of $\mathcal L_\alpha$ at each fiber of Hitchin map. 
 
\section{Preliminaries} \label{prelim}

\subsection{The moduli space of Higgs bundles}

We start by recalling the notion of Higgs bundles and the construction of their moduli space, as done in \cite{hitchinSDERS,nitsure,simpson}.

Let $\mathcal M$ be the moduli space of semistable $\SL(n,\C)$-Higgs bundles over a compact Riemann surface $X$ of genus $g$ at least $2$ and canonical bundle $K$. 
In this paper we often think of holomorphic vector bundles $E$ on $X$ as pairs $(\mathbb E,\dolb_E)$, where $\mathbb E$ is the underlying smooth vector bundle and $\dolb_E$ is the Dolbeault operator which induces the holomorphic structure. Accordingly we will denote a Higgs bundle by either $(E,\Phi)$ or $(\dolb_E,\Phi)$, depending on if we want to emphasize the role of the Dolbeault operator. Since we are working in the group $\SL(n,\C)$, the holomorphic bundle $E$ will have trivial determinant and the Higgs fields $\Phi$, which we recall is a holomorphic section of $\End E \otimes K$, will have trace zero.

We recall the construction of the moduli of Higgs bundles. We fix a smooth vector bundle $\mathbb E$ over $X$  and consider the space of pairs
\[\mathscr{H}=\{(\dolb_E,\Phi) \mid \dolb_E \Phi = 0\}.\] 
Define the slope of a vector bundle $E$ to be $\mu(E) = \deg E / \rk E$, and recall that a Higgs bundle $(E,\Phi)$ is said to be \emph{stable} if all $\Phi$-invariant subbundles $F$ satisfy $\mu(F)< \mu (E)$, \emph{semistable} if we allow equality and \emph{polystable} if it is a direct sum of stable bundles with the same slope.

The complex gauge group $\mathcal G^{\C}=\Omega^0(\Aut \mathbb E)$ acts on $\mathscr H$ by conjugation and the moduli space $\mathcal M$ may be defined as the orbit space

\[\mathscr H^{\mathrm{ps}}/\mathcal G^{\C},\] where the superscript $\mathrm{ps}$ indicates we are considering only the polystable pairs. Its smooth locus $\mathcal M^s$ is obtained by considering the orbits of the stable points in $\mathscr H$.

Now let $\mathcal N$ be the moduli space of holomorphic vector bundles over $X$ of rank $n$ and trivial determinant. Its cotangent bundle is open and dense in the moduli of Higgs bundles $\mathcal M$ since elements of the cotangent space $T_E^*\mathcal N$ of a stable bundle $E$ are precisely Higgs fields on $E$. This endows $\mathcal M$ with a natural holomorphic symplectic form $\omega$ which at the level of representatives in $T_{(E,\Phi)}\mathscr H$ is given by 

\[\omega((\beta_1,\phi_1),(\beta_2,\phi_2)) = i \int_X\tr (\phi_2 \wedge \beta_1 - \phi_1 \wedge \beta_2).\]

\subsection{\texorpdfstring{Fixed points of the $\C^*$-action}{Fixed points of the C\^*-action}}

We will consider the usual action of the group $\C^*$ on $\mathcal M$ by scaling the Higgs field
\[t \cdot (E,\Phi) = (E,t\Phi).\] All points $(E,0) \in \mathcal M$ are fixed by the $\C^*$ action and the following characterizes those fixed points with non-zero Higgs field.

\begin{prop}[{\cite[Lemma 4.1]{simpsonHBLS}}] \label{prop:fixed}		
A point $(\dolb_E,\Phi) \in \mathcal{M}$ with $\Phi \ne 0$ is fixed by the $\C^*$-action if and only if there is a smooth splitting $E= E_1 \oplus \cdots \oplus E_l$, with respect to which we have 
\[\dolb_E = \begin{pmatrix}
\dolb_{E_1} & & &  \\ 
& \ddots & &\\ 
& & \ddots & \\ 
& & &\dolb_{E_l}
\end{pmatrix}, \quad \Phi = \begin{pmatrix}
0 & & & \\ 
\Phi_1 & 0 & & \\ 
& \ddots & \ddots & \\ 
& & \Phi_{l-1} & 0
\end{pmatrix},\] where $\Phi_j : E_j \to E_{j+1} \otimes K$ is holomorphic. The vector $(\rk E_1,\ldots,\rk E_l)$ is called the \emph{type} of the fixed point.
\end{prop} 

Some words on notation: we let $\coprod_\alpha F_\alpha$ be the decomposition of the fixed point locus into its connected components and we will often denote fixed points by the single letter $\mathcal E$ for simplicity (keep in mind these are still Higgs bundles).

The splitting of Proposition \ref{prop:fixed} gives a $\Z$-grading on $\End E$ (in the setting of $\SL(n,\C)$-Higgs bundles this denotes only the traceless endomorphisms), $\End E = \bigoplus_{j \in \Z} \End_j(E)$, with 
\begin{equation} \label{end_decomp}
	\End_j(E) = \bigoplus_{i-k = j}\Hom(E_i,E_k),
\end{equation} if $1-l\le j \le l-1$ and $\End_j = 0$ for all other indices. We also set $\End_+E \coloneqq \bigoplus_{j>0}\End_j E$. We note that the grading on $\End E$ we consider here is the one in \cite{COLLIER20191193} but differs from others in the literature, for example in \cite{GPGM}, where it has the opposite sign. 

Recall that a stable $(E,\Phi) \in \mathcal M$ is a smooth point whose tangent space is given by the first hypercohomology of the deformation complex 
\begin{equation} \label{def_comp}
	\begin{tikzcd}
		C^\bullet(E,\Phi): \End E \arrow[r,"{[\Phi,-]}"] & \End E \otimes K.
	\end{tikzcd}
\end{equation}
 At a stable fixed point $(E,\Phi)$, the decomposition of Proposition \ref{prop:fixed} induces a decomposition of its tangent space (\cite[Theorem 3.8]{GPGM})
\begin{equation} \label{decomp}
	T_{(E,\Phi)}\mathcal M = \bigoplus_j \mathbb H^1(C_j^\bullet(E,\Phi)),
\end{equation}
where $C_j^\bullet(E,\Phi)$ is the subcomplex 
\begin{equation}\label{comp_decomp}
\begin{tikzcd}
C_j^\bullet(E,\Phi): \End_j E \arrow[r, "{[\Phi,-]}"] & \End_{j-1}E \otimes K.
\end{tikzcd}
\end{equation} Moreover, setting $T_{\mathcal E}^j \mathcal M \coloneqq \mathbb H^1(C_j^\bullet (\mathcal E))$, by \cite[Proposition 3.9]{GPGM} the symplectic form induces the pairing, for each $j$, 
\begin{equation}\label{weight_pair}
	(T_{\mathcal E}^j\mathcal M)^* \cong T_{\mathcal E}^{1-j}\mathcal M.
\end{equation}


We will also be dealing with two more notions related to the $\C^*$-action and its fixed points.

\begin{deff}
	Let $F_\alpha$ be a $\C^*$-fixed point component. We define the \emph{upward flow} of a fixed point $\mathcal E \in F_\alpha$ as
 \begin{equation}
	W_{\mathcal E}^+=\left\{(E,\Phi) \in \mathcal M \left| \,  \lim_{t \to 0} t \cdot (E,\Phi) = \mathcal E\right.\right\}
\end{equation}
 and the \emph{attractor} of $F_\alpha$ as
\begin{equation}
	W_\alpha^+=\left\{(E,\Phi) \in \mathcal M \left| \,  \lim_{t \to 0} t\cdot (E,\Phi) \in F_\alpha\right.\right\}.
\end{equation}
\end{deff}

\subsection{The Hitchin system}\label{hitchin_sys}

Given a basis $(a_2,\ldots,a_n)$ of the invariant polynomials of the Lie algebra $\sll(n,\C)$ with $\deg a_i = i$, we define the \emph{Hitchin map} 

\begin{equation} \label{hitchin_map}
\begin{array}{rrcl}
    h : & \mathcal M & \longrightarrow & \mathcal A \\
        & (E,\Phi) & \longmapsto     & (a_2(\Phi),\ldots,a_n(\Phi)),
\end{array}
\end{equation}
where $\mathcal A \coloneqq \bigoplus_{i=2}^n H^0(K^i)$ is a complex vector space with half the dimension of $\mathcal M$. Two examples of such bases are the coefficients of the characteristic polynomial of $\Phi$ and the traces of its powers.

For generic $a \in \mathcal A$, the fiber $h^{-1}(a)$ is (a torsor for) an abelian variety as per \cite{BNR}. In fact, to each $a \in \mathcal A$ there exists an associated \emph{spectral curve} $X_a$ sitting inside $|K|$, the total space of the canonical bundle $K$, together with a projection $\pi_a:X_a \to X$, defined by the vanishing locus of the section
\[\lambda^n + a_2 \lambda^{n-2} + \cdots + a_n,\] where $\lambda$ is the tautological section  of the pullback of $K$ on $|K|$. If $a$ is such that $X_a$ is smooth, there is a 1-1 correspondence (the so called BNR correspondence) between Higgs bundles $(E,\Phi) \in h^{-1}(a)$ and line bundles of a fixed degree $d$ on $X_a$, i.e., points of the Jacobian variety $\Jac^d(X_a)$. Roughly speaking, the points of the spectral curve over a given $c \in X$ correspond to the eigenvalues of $\Phi_c$ and the line bundle over those points corresponds to its eigenspaces.

On the other hand, the fiber over $0$ -- the \emph{nilpotent cone} -- has a complicated structure, with multiple irreducible components $\mathcal C_\alpha$, indexed by the components $F_\alpha$ of the fixed point locus of the $\C^*$-action.

The map $h$ is proper and gives $\mathcal M$ the structure of a complex integrable system (\cite{hitchinSDERS,hitchinSBIS}). Moreover it is equivariant with respect to the $\C^*$-action on $\mathcal M$ and the natural weighted action of $\C^*$ on $\mathcal A$. Together with properness, this shows that $\lim_{t \to 0} t\cdot(E,\Phi)$ always exists and is a fixed point of the action. Thus the attractors $W_\alpha^+$ partition the moduli space into its Białynicki-Birula stratification: by \cite[Theorem 4.1]{bia_bir}, $W_\alpha^+$ is locally closed and the map
\begin{equation}\label{fibration}
\pi_\alpha^+: W_{\alpha}^+ \longrightarrow F_\alpha
\end{equation}
sending a point to its $\C^*$-limit to $0$ is an affine fibration over the fixed point component $F_\alpha$ with fiber dimension equal to $\dim(T_{\mathcal E} W_{\mathcal E}^+)$. It also follows that, for a fixed point $\mathcal E \in F_\alpha$ (with non-zero Higgs field), we have
\begin{equation} \label{BBweights}
	T_{\mathcal E} F_\alpha = T^0_{\mathcal E} \mathcal M, \qquad T_{\mathcal E} W_{\mathcal E}^+ = \bigoplus_{j > 0} T^j_{\mathcal E} \mathcal M, \qquad T_{\mathcal E} \mathcal C_\alpha = \bigoplus_{j \le 0} T^j_{\mathcal E} \mathcal M
\end{equation}
where we recall that $T_{\mathcal E}^j \mathcal M =\mathbb H^1(C^\bullet_j (\mathcal E) )$, as in \eqref{decomp}. 
 
\begin{rmk}
	For any complex reductive Lie group $G$ there is a notion of $G$-Higgs bundle as a holomorphic principal bundle together with a section of an associated bundle, which also gives rise to a moduli space $\mathcal M(G)$. In this text we will be concerned only with linear Lie groups, mostly $\SL(n,\C)$ as we have described so far, but on occasion we will also mention other groups, such as $\GL(n,\C)$, the real forms $\SU(p,q)$ and $\mathrm U(p,q)$, and the projective linear group $\PGL(n,\C)=\PSL(n,\C)$. In this case we will denote the corresponding Hitchin map by $h_G:\mathcal M(G) \to \mathcal A_G$. We note here that
	\[
	\mathcal A_{\PGL} = \mathcal A, \text{ and } \mathcal A_{\GL} = H^0(K) \oplus \mathcal A.
	\]
\end{rmk}

\section{Zero velocity Lagrangians} \label{construct}

\subsection{Motivation}


It is a well known fact that one way of producing Lagrangians inside the cotangent bundle of a manifold $Y$ is to consider the conormal bundle of a submanifold $S \subset Y$. To give a very simple example in our setting, consider a point $E \in \mathcal N$. Then the conormal bundle associated to the inclusion $\{E\} \subset \mathcal N$ coincides with the cotangent fiber $H^0(\End E \otimes K)$ which embeds as a Lagrangian submanifold inside $\mathcal M$ via the map $\Phi \mapsto (E,\Phi)$.

In order to generalize this phenomenon, consider a Lagrangian $L$ of a symplectic manifold $M$ (playing the role of $Y \subset T^*Y$). Letting $S$ be a submanifold of $L$, consider the collection of all conormal spaces $N^*_p S \subset T_p ^* L$, for $p \in S$. There is now no \emph{a priori} natural way to embed the conormal bundle $N^*S$ (as a submanifold of $L$) back in $M$, so the goal becomes to find another Lagrangian $\mathcal L$ which integrates the conormal spaces at the points of $S$. In other words, we seek a Lagrangian $\mathcal L \subset M$ satisfying, for all $p \in S$, 
\begin{equation}
	T_p \mathcal L \cong T_p S \oplus N_p^* S.
\end{equation}



Going back to the moduli space of Higgs bundles, we are interested in the case where $S$ is a component $F_\alpha$ of the fixed point locus of the $\C^*$-action (with non-zero Higgs field) and $L$ is an irreducible component $\mathcal C_\alpha$ of the nilpotent cone $h^{-1}(0)$ (it follows from Proposition \ref{prop:fixed} that each such fixed point indeed has nilpotent Higgs field, so that $F_\alpha \subset \mathcal C_\alpha$). In this case, letting $\mathcal E \in F_\alpha$, the defining short exact sequence of the conormal bundle at $\mathcal E$,
\[0 \to N_{\mathcal E}^*F_\alpha \to T_{\mathcal E}^* \mathcal{C}_\alpha \to T_{\mathcal E}^* F_\alpha \to 0,\] can be written, using \eqref{BBweights}, as
\[0 \to N_{\mathcal E}^*F_\alpha\to \bigoplus_{j \le 0} \left( T^j_{\mathcal E}\mathcal M\right)^*  \to \left(T^0_{\mathcal E}\mathcal M\right)^* \to 0.\] 
By \eqref{weight_pair}, it follows that
\[N_{\mathcal E}^*F_\alpha \cong \bigoplus_{j \ge 2} T^j_{\mathcal E}\mathcal M.\] Accounting for the variation of $\mathcal E$ along $F_\alpha$, we seek a Lagrangian subvariety $\mathcal L_\alpha$ satisfying

\begin{equation}\label{weights}
	T_{\mathcal E}\mathcal L_\alpha \cong T^0_{\mathcal E}\mathcal M \oplus \bigoplus_{j \ge 2} T^j_{\mathcal E}\mathcal M,
\end{equation}
for each $\mathcal E \in F_\alpha$.

\subsection{Construction}

In the following, we will first obtain the desired Lagrangian locally around each fixed point $\mathcal E \in F_\alpha$, using tools from \cite{COLLIER20191193}. We then recognize the global condition which will allow us to glue all of the local descriptions into a single Lagrangian.

To begin, we first recall that, via the nonabelian Hodge correspondence (\cite{hitchinSDERS,donaldson_THMSDE,simpson88,10.4310/jdg/1214442469}), every stable Higgs bundle $(E,\Phi)$ has a unique hermitian metric -- called the harmonic metric -- which solves Hitchin's equation, \[F_{\nabla}+[\Phi,\Phi^*]=0.\] Here $\nabla$ is the Chern connection on $E$ associated to the harmonic metric. 
Then $F_\nabla$ is its curvature and $\Phi^*$ is the adjoint of $\Phi$ with respect to the harmonic metric. 

Setting $\partial_E \coloneqq \nabla^{1,0}$, we define the operators 
\[D'' = \dolb_E +\Phi, \quad D' = \partial_E + \Phi^*.\]

These allow us to define the following local model for $\mathcal M$ around a point $(\dolb_E,\Phi)$.

\begin{deff}\label{def:hodge}
	The \emph{Hodge slice} at a Higgs bundle $(\dolb_E,\Phi) \in \mathscr H$ is given by
	\begin{equation}\label{eq:hodge}
		\mathcal{S}_{(\dolb_E,\Phi)} \coloneqq \left\{(\beta,\phi)\in \Omega^{0,1}(\End E) \oplus \Omega^{1,0}(\End E) \left|
		\begin{array}{l}
			D''(\beta,\phi)+[\beta,\phi]=0 \\ 
			D'(\beta,\phi)=0
		\end{array}\right.\right\}.
	\end{equation}
\end{deff}

Note that, setting 
\[\mathcal H^1(E,\Phi)\coloneqq\left\{(\beta,\phi)\in \Omega^{0,1}(\End E) \oplus \Omega^{1,0}(\End E) \left|
\begin{array}{l}
 D''(\beta,\phi)=0 \\ 
 D'(\beta,\phi)=0
 \end{array}\right.\right\}\] (we are essentially picking a representative in $\ker D'$ for each element of $\mathbb H^1(\dolb_E,\Phi)$), we have \[T_{(0,0)} \mathcal S_{(\dolb_E,\Phi)} = \mathcal H^1(E,\Phi). \]


The following result tells us that the Hodge slices give local charts for the moduli space.

\begin{prop}[{\cite[Proposition 3.4]{COLLIER20191193}}] \label{prop:hodge}
If $(\dolb_E,\Phi)$ is stable, then the map
\[
\begin{array}{crcl}	
p_{(\dolb_E,\Phi)}:&\mathcal{S}_{(\dolb_E,\Phi)} \cap \mathscr{H}^{\mathrm{ps}} & \longrightarrow & \mathcal{M} \\ 
&(\beta,\phi) & \longmapsto & [(\dolb_E+\beta,\Phi+\phi)]
\end{array}	
\] is a homeomorphism from a neighborhood of the origin in $\mathcal{S}_{(\dolb_E,\Phi)}$ to a neighborhood of $[(\dolb_E,\Phi)]$ in $\mathcal M$.
\end{prop}

In \cite{COLLIER20191193}, the Hodge slice at a stable fixed point $\mathcal E$ is modified to only contain the positive directions and it is then shown that this new ``BB-slice'' $\mathcal S^+_{\mathcal E}$ is biholomorphic to the upward flow $W_{\mathcal E}^+$, thus allowing the authors to show that it is a Lagrangian submanifold of $\mathcal M$ (which was also seen in \cite{HH}).

We proceed in a similar fashion and, at a stable fixed point $\mathcal E = (E,\Phi)$, we modify the BB-slice by swapping the weights since, according to \eqref{weights}, the tangent space of our Lagrangian at a fixed point should have weights $0$ and greater or equal to $2$. Thus, we ``remove weight 1'' and ``include weight 0'' in our new slice:
\begin{equation}\label{new_slice}
	\widehat{\mathcal S}_{\mathcal E}\coloneqq\left\{ \! \left.(\beta,\phi)\in
	\begin{array}{c}
		\Omega^{0,1}(\End_{0}E \oplus \End_{\ge 2}E) \\ \oplus \\  \Omega^{1,0}(\End_{-1} E \oplus \End_{\ge 1}E)
	\end{array}
	\! \right| \!
	\begin{array}{l}
		D''(\beta,\phi)+[\beta,\phi]=0 \\ 
		D'(\beta,\phi)=0
	\end{array} \! \! \right\}.
\end{equation}

One issue that arises when performing this swap is that the argument of  \cite[Corollary 4.3]{COLLIER20191193} showing that the map $p_{\mathcal E}$ restricted to the BB-slice is injective breaks down with the inclusion of weight $0$. We then have to restrict the map to a neighborhood of the origin of the slice, obtaining an isotropic patch (which we will later see is in fact Lagrangian) around each fixed point satisfying \eqref{weights}.

 \begin{prop}\label{local_lag}
 	Let $\mathcal E$ be a stable fixed point of the $\C^*$-action and let $U$ be a neighborhood of the origin of the Hodge slice $\mathcal S_{\mathcal E}$ such that the map $p_{\mathcal E}$ of Proposition \ref{prop:hodge} is a biholomorphism between $U$ and its image. Then
 	\[\mathcal L_{\mathcal E}^U \coloneqq p_{\mathcal E}(\widehat{\mathcal S}_{\mathcal E}\cap U)\] is isotropic in $\mathcal M$. 
 \end{prop}

 \begin{proof}
 	
 	Recalling \eqref{eq:hodge}, the tangent space to the Hodge slice at a fixed point $\mathcal E = (\dolb_E,\Phi)$ is given by  
\[T_{(\beta,\phi)}\mathcal S_{\mathcal E} = \{(\dot\beta,\dot\phi) \in \Omega^{0,1}(\End E)\oplus \Omega^{1,0}(\End E)\mid D''_{(\beta,\phi)}(\dot\beta,\dot\phi)=0,\, D'(\dot\beta,\dot\phi)=0\},\] where $D''_{(\beta,\phi)}=D'' + \beta +\phi = (\dolb_E +\beta)+(\Phi+\phi)$. This computation essentially uses that the operators $D'$ and $D''$ are linear but the commutator is quadratic.

The derivative of $p_{\mathcal E}$ at a point $(\beta,\phi)  \in \mathcal S_{\mathcal E}$ is given by

\[
\begin{array}{crcl}
d_{(\beta,\phi)}p_{\mathcal E}: & T_{(\beta,\phi)}\mathcal S_{\mathcal E} & \longrightarrow & T_{(\dolb_E+\beta,\Phi+\phi)}\mathcal M \cong \mathbb H^1(\dolb_E+\beta,\Phi+\phi) \\ 
& (\dot \beta,\dot \phi) & \longmapsto & [(\dot \beta,\dot \phi)].

\end{array}
\] If we again identify the target tangent space with $\mathcal H^1(\dolb_E + \beta,\Phi+\phi)$, we get an isomorphism \[\ker D''_{(\beta,\phi)} \cap \ker D' \cong \ker D''_{(\beta,\phi)} \cap \ker D'_{(\beta,\phi)}, \] i.e. the derivative of $p_{\mathcal E}$ essentially just picks the ``right'' harmonic representative (relative to the metric of the pair $(\dolb_E+\beta,\Phi+\phi)$) for the cohomology class $[(\dot \beta,\dot \phi)]$. 

Using our modified slice $\widehat{\mathcal S}_{\mathcal E}$, defined in \eqref{new_slice}, we have a description of the tangent space to $\mathcal L_{\mathcal E}^U$ as 
\begin{equation} \label{slice_tangent}
	\begin{aligned}
		T_{(\dolb_E+\beta,\Phi+\phi)}\mathcal L_{\mathcal E}^U &= d_{(\beta,\phi)}p_{\mathcal E}(T_{(\beta,\phi)}\widehat{\mathcal S}_{\mathcal E}) \\[1ex]
		&\cong \left\{\left.(\dot\beta,\dot\phi) \in 
		\begin{array}{c}
			\Omega^{0,1}\left(\End_0 E \oplus \bigoplus_{i \ge 2}\End_i E\right) \\ 
			\oplus \\ 
			\Omega^{1,0}\left(\End_{-1} E \oplus \bigoplus_{i \ge 1}\End_i E\right)
		\end{array}
		\right| 
		\begin{array}{l}
			D''_{(\beta,\phi)}(\dot \beta, \dot \phi)=0 \\ 
			D'_{(\beta,\phi)}(\dot \beta,\dot\phi)=0
		\end{array}
		\right\}. 
	\end{aligned}
\end{equation}


Recalling that the symplectic form on $\mathcal M$ pairs weights $j$ and $1-j$ (\cite[Proposition 3.9]{GPGM}), it follows that this space is isotropic. \end{proof}


We now need to show that, given a component $F_\alpha$, all of these patches glue together to a global well-defined object which integrates the conormal spaces. The way we will accomplish this is by showing that these Lagrangian pieces are in fact charts for a subvariety of the attractor $W_\alpha^+$. It will consist of points which "arrive at their limit to $0$ with zero velocity". In order to be more precise, let $F_\alpha$ be a fixed point component and, for $(E,\Phi) \in W_\alpha^+$, define a curve $\gamma_{(E,\Phi)}: \C \to \mathcal{M}$ by 

\[
\begin{array}{crcl}
\gamma_{(E,\Phi)}: & \C & \longrightarrow & \mathcal{M} \\ 
& 0 \ne t &\longmapsto & t \cdot (E,\Phi) \\ 
& 0 & \longmapsto & \lim_{t \to 0} t \cdot (E,\Phi).
\end{array}
\]

This allows us to define the following (in the course of the proof of Proposition \ref{zv_lag}, we will see that it makes sense to take the derivative of $\gamma_{(E,\Phi)}$ at $t=0$).

\begin{deff}
	Let $F_\alpha$ be a fixed point component, we define the \emph{zero velocity locus} of its attractor $W_{\alpha}^+$ as
	\[\mathcal L_\alpha \coloneqq \{(E,\Phi) \in W_\alpha^+ \mid \gamma_{(E,\Phi)}'(0)=0\}.\]
\end{deff}

\begin{prop}\label{zv_lag}
Let $F_\alpha$ be a $\C^*$-fixed point component. Then, for each fixed point $\mathcal E \in F_\alpha$ and each neighborhood $U$ of $\mathcal E$ in the conditions of Proposition \ref{prop:hodge}, we have
	\[\mathcal L_\alpha \cap p_{\mathcal E}(U) = \mathcal L_{\mathcal E}^U.\]
\end{prop}


\begin{proof}

Letting $\mathcal E = (\dolb_E,\Phi)$, a point in $p_{\mathcal E}(U)$ is of the form 
	\[(\dolb_E+\beta,\Phi+\phi),\]
with 
$(\beta,\phi) \in \Omega^{0,1}\left( \bigoplus_{i \ge 0}\End_i E\right)\oplus \Omega^{1,0}\left( \bigoplus_{i \ge -1}\End_i E\right)$.
Denoting by $\gamma_{(\beta,\phi)}$ the curve starting at $p_{\mathcal E}(\beta,\phi)$, we have, using the gauge transformation in \cite[Proposition 4.2]{COLLIER20191193},
\[
\gamma_{(\beta,\phi)}(t)= (\dolb_E+\beta,t\Phi+t\phi) \cong \left(\dolb_E + \beta_0 + \sum_{j \ge 1}t^j\beta_j,\Phi + \phi_{-1} + \sum_{j\ge 0}t^{j+1} \phi_j\right). 
\] Note that at $t=0$, we have \[\gamma_{(\beta,\phi)}(0)=(\dolb_E+\beta_0,\Phi + \phi_{-1})=\lim_{t\to 0} t\cdot (\dolb_E + \beta,\Phi + \phi).\]

Taking the derivative at $t=0$ we get
\[\gamma'_{(\beta,\phi)}(0)=(\beta_1,\phi_0),\] from where it follows that, recalling \eqref{new_slice},
\[\gamma'_{(\beta,\phi)}(0)=0 \iff (\beta,\phi) \in \widehat{\mathcal S}_{(\beta,\phi)}.\qedhere\] \end{proof} 



\begin{theorem}
	\label{thm:A}
	\ThmA
\end{theorem}

\begin{proof}
	By Propositions \ref{local_lag} and \ref{zv_lag}, it remains to show that $\mathcal L_\alpha$ is half dimensional.
	
	Recall the affine fibration given by \eqref{fibration}. Clearly, $\mathcal L_\alpha \subset W_{\alpha}^+$ and for a fixed point $\mathcal E \in F_\alpha$, we have
	\[\mathcal L_\alpha \cap W_{\mathcal E}^+ = p_{\mathcal E}(\widehat{S}_{\mathcal E}\cap \mathcal S^+_{\mathcal E}) \subset p_{\mathcal E}(\mathcal S^+_{\mathcal E}) = W_{\mathcal E}^+.\] Since $\mathcal L_{\mathcal E} \coloneqq \mathcal L_\alpha \cap W_{\mathcal E}^+$ is $\C^*$-invariant, it follows that $\mathcal L_\alpha$ is an affine subfibration of $W_{\alpha}^+$. 
	
Thus to compute the dimension of $\mathcal L_\alpha$, we use 
	\begin{equation} \label{dim_L}	
		\dim \mathcal L_\alpha = \dim \mathcal L_{\mathcal E} + \dim F_\alpha. 
	\end{equation}

	We use the biholomorphism in \cite[Theorem 3.9]{COLLIER20191193} between the BB-slice $\mathcal S^+_{\mathcal E}$ and the positive weight part of the tangent space $\mathbb H^1_+(\mathcal E)$, given by the Kuranishi map. Since the map is equivariant, it restricts to a biholomorphism betweent the parts with weight strictly bigger then $1$, i.e., $\widehat{\mathcal S}_{\mathcal E} \cap \mathcal S^+_{\mathcal E}$ and $\bigoplus_{j>1} \mathbb H^1_j(\mathcal E)$. Thus, equation \eqref{dim_L} becomes
	\begin{equation*}
		\dim \mathcal{L}_\alpha = \dim\bigl(\widehat{\mathcal{S}}_{\mathcal{E}} \cap \mathcal{S}^+_{\mathcal{E}}\bigr) + \dim F_\alpha \nonumber 
		= \sum_{j > 1} \dim\bigl(\mathbb{H}^1_j(\mathcal{E})\bigr) + \dim\bigl(\mathbb{H}^1_0(\mathcal{E})\bigr) \nonumber 
		= \dim \bigl(W_{\mathcal{E}}^+\bigr).
	\end{equation*}
	
	Since the upward flow $W_{\mathcal E}^+$ is half-dimensional by \cite[Lemma 4.1]{COLLIER20191193}, it follows that the same is true for $\mathcal L_\alpha$. \qedhere
	
\end{proof}

\begin{figure}[htbp]
	\centering
	\begin{subfigure}[b]{0.48\textwidth}
		\centering
		\includegraphics[width=\textwidth]{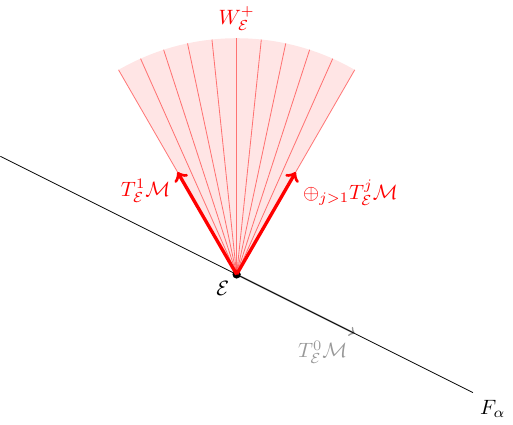} 
		\caption{The upward flow at a fixed point $\mathcal E \in F_\alpha$.}
		\label{fig:left_image}
	\end{subfigure}
	\hfill 
	\begin{subfigure}[b]{0.48\textwidth}
		\centering
		\includegraphics[width=\textwidth]{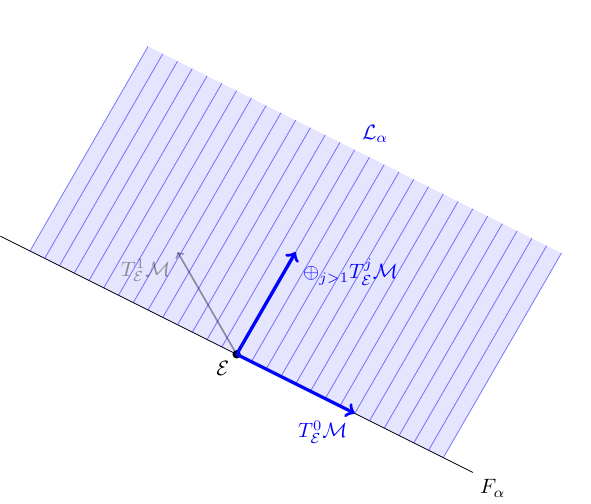} 
		\caption{Lagrangian $\mathcal L_\alpha$ around a fixed point $\mathcal E \in F_\alpha$.}
		\label{fig:right_image}
	\end{subfigure}
	
	\caption{Comparison between an upward flow $W^+_\mathcal E$ and the Lagrangian $\mathcal L_\alpha$.}
	\label{fig:global_label}
\end{figure}

\begin{example}\label{SU_pq}
The simplest example of non-trivial fixed points in rank $n$ are the ones of type $(p,q)$, with $p+q=n$ and $p,q \ge 1$, i.e., direct sums $E_0 \oplus F_0$, with $\rk E_0 = p$, $\rk F_0 = q$ and Higgs field $\Phi:E_0 \to F_0K$. A component $F_\alpha$ of fixed points of this type is obtained by fixing the degree of $E_0$ and $F_0$ and in this case the Lagrangian $\mathcal L_\alpha$ is made of points of the form
\[\left(E \oplus F, \begin{psmallmatrix}
	0 & \beta \\
	\gamma & 0 
\end{psmallmatrix} \right),\] with $\deg E = \deg E_0$, $\deg F = \deg F_0$, $\gamma \in H^0(E^*FK)$ and $\beta \in H^0(F^* EK)$. These are precisely $\SU(p,q)$-Higgs bundles sitting inside $\mathcal M$. 
\end{example}

\begin{rmk}\label{not_closed}
	The Lagrangian $\mathcal L_\alpha$ is closed in $W_\alpha^+$ but since the latter is only locally closed, in general $\mathcal L_\alpha$ is not closed in $\mathcal M$. For instance, in Example \ref{SU_pq}, the Lagrangians $\mathcal L_\alpha$ we construct contain in their closure Higgs bundles flowing to fixed points of other types other than $(p,q)$.
	
	Moreover, we will later see in Section \ref{z_nilpotents} that $\mathcal L_\alpha$ can intersect the nilpotent cone non-trivially, i.e., at a non-fixed point, and so, being $\C^*$-invariant, it will contain fixed points of other components in its closure. 

\end{rmk}

\subsection{The Simpson filtration}

 Given a point $(\dolb_E,\Phi) \in \mathcal M$, we can use the smooth splitting of its limit point to $0$ (described in Proposition \ref{prop:fixed}) and Proposition \ref{zv_lag} to conclude that it is in $\mathcal L_\alpha$ if and only if can be represented as

\begin{equation} \label{zv_slice}
\dolb_E = \begin{pmatrix}
	\dolb_{E_1} & 0 & * &  *  \\ 
	& \ddots & \ddots & *\\ 
	& & \ddots & 0 \\ 
	& & &\dolb_{E_l}
\end{pmatrix}, \quad \Phi = \begin{pmatrix}
	0 & * & * &  *\\ 
	\Phi_1 & 0 & *&* \\ 
	& \ddots & \ddots & *\\ 
	& & \Phi_{l-1} & 0
\end{pmatrix}.
\end{equation}
From this representation, we can immediately see what is the limit of $t\cdot (\dolb_E,\Phi)$ as $t$ goes to $0$. It will however be useful to be able to identify the limit without resorting to the slice. To do that, we need the following concept.

\begin{deff}[\cite{simpsonIDMVBC}]
	Let $(E,\Phi)$ be a Higgs bundle on $X$. A filtration $E_\bullet$ of $E$ by strict subbundles
	
	\[0 = E_0 \subset E_1 \subset \cdots \subset E_k = E\] is said to be a \emph{Simpson filtration} if the following are satisfied:
	\begin{enumerate}
		\item (Griffiths transversality) for each $i=0,\ldots,k-1$, $\Phi(E_i)\subset E_{i+1} K$;
		\item (graded semistability) the associated graded Higgs bundle,
		\[\left(\bigoplus_{i=1}^{k} E_i/E_{i-1}, \overline{\Phi}\right),\] is semistable, where $\overline{\Phi}$ is the induced Higgs field satisfying $\overline{\Phi}(E_i/E_{i-1})\subset E_{i+1}/E_{i}K$.  
	\end{enumerate}
\end{deff} 

The role of the Simpson filtration in determining the limit point of the $\C^*$-action lies in the following result, first shown in the context of flat bundles in \cite{simpsonIDMVBC} and later for Higgs bundles in \cite[Proposition 4.2]{COLLIER20191193} and \cite[Proposition 3.4]{HH}. 

\begin{theorem}
	Let $(E,\Phi) \in \mathcal M$. Then 
	\[\lim_{t \to 0} t \cdot (E,\Phi) = \mathcal E\] if and only if $\mathcal E$ is isomorphic to the associated graded of the Simpson filtration of $(E,\Phi)$. 
\end{theorem}

A natural question that arises is then how to detect if a Higgs bundle is in $\mathcal L_\alpha$ via its Simpson filtration. Analyzing \eqref{zv_slice}, we get the following characterization.

\begin{prop}\label{zv_simp}
	Let $(E,\Phi) \in \mathcal M$, with $E_\bullet$ its Simpson filtration and $(\bigoplus_i F_i,\overline \Phi)$ the corresponding associated graded object. Then $(E,\Phi) \in \mathcal L_\alpha$ if and only if (setting $F_0 = E_{-1} = 0$)
	
	\begin{enumerate}
		\item (bundle conditions) for all $i\ge 1$, we have 
			$E_i/E_{i-2} \cong F_{i-1}\oplus F_i;$
		equivalently the following sequence
		\begin{equation} \label{zv_seq}
			\begin{tikzcd}
				0 \arrow[r] & E_{i-2} \arrow[r] & E_i \arrow[r] & F_{i-1} \oplus F_i \arrow[r] & 0
			\end{tikzcd}
		\end{equation} is exact;
		
		\item (field conditions) 
		for each $i \ge 1$, the composition
		\begin{equation}\label{zv_higgs}
				\begin{tikzcd}
				F_{i} \arrow[r] & E_i/E_{i-2} \arrow[r,"\Phi^*_i"] & E_{i+1}/E_{i-1} \arrow[r] & F_i 
			\end{tikzcd}
		\end{equation}
		is the zero map, where $\Phi_i^*$ is the natural map induced by $\Phi$ and the remaining ones are the natural injection and projection associated to the direct sum in \eqref{zv_seq}. 
		
	\end{enumerate}
\end{prop}

\section{\texorpdfstring{Lagrangian for type $(1,\ldots,1)$ fixed point components}{Lagrangian for type (1,...,1) fixed point components}} \label{main}


From now on we will specialize to Lagrangians $\mathcal L_\alpha$ associated to fixed point components $F_\alpha$ of type $(1,\ldots,1)$, i.e., sums of line bundles, and how they interact with the fibers of the Hitchin map $h$ in \eqref{hitchin_map}. 

We first review, in subsection \ref{4_1}, the basics on fixed points of type $(1,\ldots,1)$. In subsection \ref{4_2}, we introduce our main technical tool, Hecke transformations, and how they interact with upward flows.

In subsection \ref{z_nilpotents} we will consider the intersection of the Lagrangian and the nilpotent cone---the fiber over 0---namely we will characterize the fibers $\mathcal L_{\mathcal E}$ which intersect the nilpotent cone non-trivially (i.e., at a point other than the fixed point). 

Later, in section \ref{int_generic} we will be concerned with the generic fibers of $h$. We will see that they intersect the Lagrangian $\mathcal L_\alpha$ in a finite number of points, which will be computed. Knowing that number, we then propose a candidate in section \ref{mirror} for what the mirror BBB-brane should be in the dual moduli space. 

\subsection{Preliminaries: fixed points of type $(1,\ldots,1)$} \label{4_1}

A fixed point of type $(1,\ldots,1)$ is of the form \[(L_1 \oplus \cdots \oplus L_n,\Phi)\] where each $L_i$ is a line bundle over $X$ and  \[\Phi_i = \Phi|_{L_i} \in H^0(L_i^*L_{i+1}K).\] We will also use the following chain notation to denote such a fixed point (where the arrows are to be interpreted as $K$-twisted maps):
\begin{equation}\label{chain}
	\begin{tikzcd}
		L_1 \arrow[r,"\Phi_1"] & L_2 \arrow[r,"\Phi_2"] & \cdots \arrow[r,"\Phi_{n-2}"] & L_{n-1} \arrow[r,"\Phi_{n-1}"] & L_n.
	\end{tikzcd}
\end{equation}

 Following \cite[Section 3.1]{HH}, if we write 
 \begin{equation}\label{divisors}
 	 L_1 = \mathcal O(\mathcal \delta_0) \text{ and } \delta_i = \div \Phi_i,
 \end{equation} each fixed point is determined by a choice of divisors $(\delta_0,\delta_1,\ldots,\delta_{n-1})$, where all but the first are required to be effective. Since we are fixing the determinant, a component $F_\alpha$ is obtained by fixing the degrees $m_i \coloneqq \deg \delta_i$, for $i\ge 1$, and we get a degree $n^{2g}$ covering of the product 
\[X^{[m_1]}\times \cdots \times X^{[m_{n-1}]},\] where $X^{[k]}$ denotes the $k$-th symmetric product of points of $X$. Indeed, after fixing the divisors $\delta_1$ through $\delta_{n-1}$, we can write $L_i = K^{-(i-1)}(\delta_0+\cdots+ \delta_{i-1})$, for $i\ge 1$, which allows us to rewrite the equation $L_1\cdots L_n \cong \mathcal O$ as
\[
L_1^n \cong K^{(n-1)n/2}(D),
\] where $D = -\sum_{i=1}^{n-1} (n-i)\delta_i$. Thus, solving for $L_1 = \mathcal O(\delta_0)$ involves taking the $n$-th root of a line bundle, which gives rise to $n^{2g}$ possible solutions (for a more detailed account, see the proof of \cite[Proposition 10.1]{HT}). 

 We will also write 
 \begin{equation}\label{delta_plus}
 	\delta_+ \coloneqq \delta_1 + \cdots + \delta_{n-1}
 \end{equation}
  and $\deg \Phi_i$ will mean $\deg \delta_i$.

\begin{example}\label{hitchin_sec}
	Setting each $m_i$ equal to $0$, we can assume $\Phi_i = 1$ and we obtain $n^{2g}$ fixed points of the form
	\[
	\mathcal E_L = \left( E_L = L \oplus L  K^{-1} \oplus \dots \oplus L K^{1-n} , \,
	\begin{pmatrix}
		0 &        &        &   \\
		1 & 0      &        &   \\
		& \ddots & \ddots &   \\
		&        & 1      & 0
	\end{pmatrix} \right)
	\] with $L^n \cong K^{n(n-1)/2}$. Thus in this case each fixed point component is reduced to a single point $F_\alpha = \{\mathcal E_L\}$. In particular we have $T^0_{\mathcal E_L}\mathcal M = T^1_{\mathcal E_L}\mathcal M = 0$ so that
	\[\mathcal L_\alpha = \mathcal L_{\mathcal E_L} = W_{\mathcal E_L}^+,\] i.e. the Lagrangian in this case coincides with the upward flow of $\mathcal E_L$. Such an upward flow is called a \emph{Hitchin section} and can be fully described as follows.
	
	Let $a=(a_2,\ldots,a_n)$ be a point in the Hitchin base $\mathcal A$ and consider the Higgs field given by the companion matrix 
	\[\Phi_a=
	\begin{pmatrix}
		0 & 0 & \dots & 0 & -a_n \\
		1 & 0 & \dots & 0 & -a_{n-1} \\
		0 & 1 & \ddots & 0 & -a_{n-2} \\
		\vdots & \vdots & \ddots & \ddots & \vdots \\
		0 & 0 & \dots & 1 & 0
	\end{pmatrix}.\] The assigment $a \mapsto ( E_L,\Phi_a)$ defines a section of the Hitchin map \eqref{hitchin_map} whose image is precisely $W_{\mathcal E_L}^+$ (shown in \cite{HIT3}). 
\end{example}

\subsection{Hecke transformations} \label{4_2}

One of the main technical tools we will use in Section \ref{main} will be that of a Hecke transformation of a Higgs bundle. In this subsection, we will summarize the main definitions and results. 

\begin{deff}
	Let $E$ be a holomorphic vector bundle over a Riemann surface $X$, let $c \in X$ and $V$ a subspace of the fiber $E_c$. The \emph{Hecke transform} $\mathcal H_V(E)$ of $E$ at $V$ is defined by the following short exact sequence of sheaves,
	\begin{equation}
		\begin{tikzcd}
			0 \arrow[r] & \mathcal H_V(E) \arrow[r] & E \arrow[r] & (E_c/V)_c \arrow[r] & 0, 
		\end{tikzcd}
	\end{equation} where by $U_c$, when $U$ is a vector space, we mean the torsion sheaf $U \otimes \mathcal O_c$. 
\end{deff}

The sheaf $\mathcal H_V(E)$ is the sheaf of sections of a vector bundle, again denoted by $\mathcal H_V(E)$ (or $E_V$ or even $E'$ when $V$ is understood). It can be interpreted as the subsheaf of sections of $E$ which at $c$ take values in $V \subseteq E_c$. In particular, if $V=0$, we get the more familiar object $E(-c)= E \otimes \mathcal O(-c)$ which fits in the exact sequence
\begin{equation}
	\begin{tikzcd}
		0 \arrow[r] & E(-c) \arrow[r] & E \arrow[r] & E_c \arrow[r] & 0.
	\end{tikzcd}
\end{equation} At the other extreme, if $V = E_c$, the exact sequence describes an isomorphism
\begin{equation}
	\begin{tikzcd}
		0 \arrow[r] & \mathcal H_V(E) \arrow[r] & E \arrow[r] & 0, 
	\end{tikzcd}
\end{equation} as expected, since there are no restrictions on the sections of $\mathcal H_V(E)$ in this case.

\begin{example}
	An important case we will be using is when $E$ has a direct sum decomposition $E = F_1 \oplus \cdots \oplus F_k$ and $V = \bigoplus_{i \in \Sigma} (F_i)_c$, where $\Sigma \subset \{1,\ldots,k\}$. Then $\mathcal H_V(E)=F'_1 \oplus \cdots \oplus F'_k$, with
	\begin{equation}
		F'_i = \begin{cases}
			F_i \text{ if } i \in \Sigma \\ 
			F_i(-c) \text{ if } i \notin \Sigma.
		\end{cases}
	\end{equation}
\end{example}

Now let $(E,\Phi)$ be a Higgs bundle over $X$, let $c \in X$ and consider $V \subseteq E_c$ a $\Phi_c$-invariant subspace. Letting $E'$ be the Hecke transform of $E$ at $V$, we get a commutative diagram
\begin{equation} \label{Hecke_higgs}
	\begin{tikzcd}
		0 \arrow[r] &  E' \arrow[r] \arrow[d,"\Phi'"] & E \arrow[r] \arrow[d,"\Phi"] & (E_c/V)_c \arrow[r] \arrow[d,"\overline \Phi"] & 0 \\ 
		0 \arrow[r] &  E'K \arrow[r] & EK \arrow[r] & (E_c/V)_c K \arrow[r] & 0 
	\end{tikzcd}
\end{equation} where $\Phi'$ and $\overline \Phi$ are the maps naturally induced by $\Phi$.

\begin{deff}
	Let $(E,\Phi)$ be a Higgs bundle over $X$, let $c \in X$ and $V \subseteq E_c$ be $\Phi_c$-invariant. Then the \emph{Hecke transform} of $(E,\Phi)$ at $V$, denoted by $\mathcal H_V(E,\Phi)$, is the Higgs bundle $(E',\Phi')$ induced by diagram \eqref{Hecke_higgs}.
\end{deff}

\begin{rmk}
	The Hecke transformed Higgs field $\Phi'$ of a Higgs bundle $(E,\Phi)$ agrees with $\Phi$ outside a point $c \in X$, so it has the same characteristic polynomial. In other words, we have $h(E,\Phi)=h(E',\Phi')$.
\end{rmk}

Now we specify to the case of type $(1,\ldots,1)$ fixed points. The following result due to \cite{HH} describes what happens to the upward flow $W_{\mathcal E}^+$, where $\mathcal E$ is a type $(1,\ldots,1)$ fixed point, when we perform a Hecke transformation at a point $c \in X$ satisfying $\Phi_1\cdots\Phi_{n-1}(c) \ne 0$.

\begin{prop}{\cite[Proposition 4.15]{HH}} \label{HH_hecke}
	Let $(E,\Phi) \in W_{\mathcal E}^+$, where $\mathcal E$ is a type $(1,\ldots,1)$ fixed point written in chain form as
	\[
	\begin{tikzcd}
		L_1 \arrow[r,"\Phi_1"] & L_2 \arrow[r,"\Phi_2"] & \cdots \arrow[r,"\Phi_{n-2}"] & L_{n-1} \arrow[r,"\Phi_{n-1}"] & L_n,
	\end{tikzcd}
	\] such that $\Phi_1 \cdots \Phi_{n-1}$ does not vanish at $c \in X$. Let $V \subset E_c$ be a $(n-i)$-dimensional and $\Phi_c$-invariant subspace. Then the Hecke transformed Higgs bundle $\mathcal H_V(E,\Phi)$ is in $W_{\mathcal E'}^+$, where $\mathcal E'$ is the type $(1,\ldots,1)$ fixed point given by
	 \[
	 \begin{tikzcd}
	 	L_1(-c) \arrow[r,"\Phi_1"] & \cdots  \arrow[r,"\Phi_{i-1}"] & L_i(-c) \arrow[r,"\Phi_{i}s_c"] & L_{i+1} \arrow[r,"\Phi_{i+1}"] & \cdots \arrow[r,"\Phi_{n-1}"] & L_n.
	 \end{tikzcd}
	 \] Here $s_c \in H^0(\mathcal O_X(c))$ is a defining section.
\end{prop} 

The following result says that, given a condition on the subspace $V$ at which the Hecke transformation is performed, the zero velocity locus $\mathcal L_{\mathcal E}$ of the initial upward flow is sent to the zero velocity locus of the target $\mathcal L_{\mathcal E'}$.

\begin{prop}\label{zv_hecke}
	Let $(E,\Phi) \in \mathcal L_{\mathcal E}$, where $\mathcal E$ is a type $(1,\ldots,1)$ fixed point written in chain form as
	\[
	\begin{tikzcd}
		L_1 \arrow[r,"\Phi_1"] & L_2 \arrow[r,"\Phi_2"] & \cdots \arrow[r,"\Phi_{n-2}"] & L_{n-1} \arrow[r,"\Phi_{n-1}"] & L_n,
	\end{tikzcd}
	\] such that $\Phi_1 \cdots \Phi_{n-1}$ does not vanish at $c \in X$. Let $V \subset E_c$ be a $(n-i)$-dimensional and $\Phi_c$-invariant subspace such that 
	\[\tr \Phi_c|_V = 0.\] 
	Then the Hecke transformed Higgs bundle $\mathcal H_V(E,\Phi)$ lies in $\mathcal L_{\mathcal E'}$, where $\mathcal E'$ is the type $(1,\ldots,1)$ fixed point given by
	\[
	\begin{tikzcd}
		L_1(-c) \arrow[r,"\Phi_1"] & \cdots  \arrow[r,"\Phi_{i-1}"] & L_i(-c) \arrow[r,"\Phi_{i}s_c"] & L_{i+1} \arrow[r,"\Phi_{i+1}"] & \cdots \arrow[r,"\Phi_{n-1}"] & L_n.
	\end{tikzcd}
	\]
\end{prop} 

\begin{proof}

	To check that the Hecke transform preserves the zero velocity condition on the bundle side, we need to check that it preserves sequence \eqref{zv_seq} of Proposition \ref{zv_simp}. This is equivalent to having, for each $k=2,\ldots,n$, commutative exact diagrams of sheaves of the form
	\[
	\begin{tikzcd}
		&0\arrow[d] & 0\arrow[d] & 0\arrow[d] \\ 
		0\arrow[r]& E'_{k-2} \arrow[r] \arrow[d] & E' \arrow[r] \arrow[d] & L'_{k-1} \oplus L'_k \arrow[d]\arrow[r] & 0\\ 
		0\arrow[r]&E_{k-2}\arrow[r] \arrow[d] & E_k \arrow[r,"p"] \arrow[d] & L_{k-1}\oplus L_k \arrow[d]\arrow[r] & 0\\ 
		0\arrow[r] & (E_{k-2})_c/V_{k-2} \arrow[r] \arrow[d] & (E_k)_c/V_k \arrow[r] \arrow[d] & (L_{k-1}\oplus L_k)_c/p_c(V_k) \arrow[d]\arrow[r] & 0 \\ 
		& 0 & 0 & 0, 
	\end{tikzcd}
	\]where $E_\bullet$ is the Simpson filtration of $(E,\Phi)$, $E'_\bullet$ is the Simpson filtration of the Hecke transform $\mathcal H_V(E,\Phi)$, $V_\bullet = V \cap (E_\bullet)_c$ and $\bigoplus L'_k$ is the associated graded of $E'_\bullet$. By Proposition \ref{HH_hecke} we have \[ L'_k = \begin{cases}
		L_k(-c) \text{ if } k \le i \\ 
		L_k \text{ if } k > i.
	\end{cases} \] 
	The only non-trivial check is the exactness of the third column near the point $c$. It is equivalent to having
	
	\begin{itemize}
		\item $p_c(V_k)=0$ if $k\le i$,
		\item $p_c(V_{i+1}) = (L_{i+1})_c$,
		\item $p_c(V_k) = (L_{i-1}\oplus L_i)_c$ if $k \ge i+2$. 
	\end{itemize}
	
	By \cite[Lemma 4.14]{HH}, we have 
	\begin{equation}\label{dim_V}
		\dim V_k = \max\{0,k-i\}. 
	\end{equation}
	
	If $k < i+1$, by \eqref{dim_V} we have $V_k = 0$, thus $p_c(V_k)$ is also zero. Again by \eqref{dim_V}, if $k \ge i+2$, then $\dim p_c(V_k)=2$, thus $p_c(V_k) = (L_{i-1}\oplus L_i)_c$. However, if $k=i+1$, \eqref{dim_V} only tells us that $V_{i+1}$ is $1$-dimensional and embeds in $(L_{i}\oplus L_{i+1})_c$ via $p_c$ and we need to see that the image of this embedding is precisely $(L_{i+1})_c$. 
	

	
	We are led to the following linear algebra problem. Consider the linear map $A \coloneqq \Phi_c : \C^n \to \C^n$, where we have chosen a basis $(e_i)_i$, where $(L_i)_c = \langle e_i \rangle$, so that $A$ is represented by the matrix
	
	\[\begin{pmatrix}
		0 & & & & & \\ 
		\Phi_1(c) & 0 & &  & * & \\ 
		& \Phi_2(c) & 0 & & & \\ 
		& & \ddots &  0 & & \\ 
		& 0 & & \Phi_{n-2}(c) & 0 & \\ 
		& & & & \Phi_{n-1}(c) & 0 \\ 
		
	\end{pmatrix},\] with $\Phi_i(c)\ne 0$. By \eqref{dim_V}, the intersection $V \cap \langle e_1,\ldots,e_{i+1}\rangle$ is 1-dimensional, spanned by a vector $v=(\beta_1,\ldots,\beta_i,\beta_{i+1},0,\ldots,0)$. Note that if $\beta_{i+1}$ were 0, then in fact $v \in V \cap \langle e_1,\ldots,e_{i}\rangle = 0$ by \eqref{dim_V}, a contradiction. Thus we can assume \[v = (\beta_1,\ldots,\beta_i,1,0,\ldots,0).\] The projection to the subspace $\langle e_i,e_{i+1}\rangle$ has coordinates $(\beta_i,1)$ and so our goal is to show that $\beta_i = 0$.
	
	Now, since $v \in V$ and $\tr A|_V=0$, by the Cayley-Hamilton theorem it must satisfy the equation
	\[A^{n-i}(v) + b_2 A^{n-i-2}(v) + \cdots + b_{n-i}v = 0,\] where $\lambda^{n-i} + b_2\lambda^{n-i-2} + \cdots + b_n$ is the characteristic polynomial of $A|_{V}$. We will look at the last coordinate of the equation above. Since $A^{n-i-j}(v) \in \langle e_1,\ldots,e_{n-(j-1)} \rangle$ for $j \ge 2$, it follows that only $A^{n-i}(v)$ has a non-zero last coordinate. To compute this entry, we first note that, for $j=1,\ldots,n-i$, the $(i+j)$-th coordinate of $A^j(v)$ equals $\beta_i \Phi_i(c) \cdots \Phi_{i+j-1}(c)$ (easily seen by induction), thus for $j=n-i$, this means that the $n$th coordinate of $A^{n-i}(v)$ is $\beta_i \Phi_i(c) \cdots \Phi_{n-i}(c)$. Since $\Phi_i(c) \cdots \Phi_{n-i}(c) \ne 0$, we have $\beta_i = 0$.
	
	Thus, we have seen that performing a Hecke transformation at an invariant subspace $V$ where $\Phi_c$ is traceless maintains the zero velocity condition on the bundle side. To check that it is preserved on the Higgs field side as well, we look again at Proposition \ref{zv_simp}. We need to check that after performing a Hecke transformation, the composition in \eqref{zv_higgs} remains zero. 
	
	
	 We have the following diagram relating \eqref{zv_higgs} for the original Higgs bundle and its Hecke transform
	\begin{equation*}
		\begin{tikzcd}
			L_{k} \arrow[r] & E_k / E_{k-2} \arrow[r,"\Phi"] & (E_{k+1}/E_{k-1})K  \arrow[r] & L_kK \\ 
			L'_{k} \arrow[r] \arrow[u,hook] & E'_k / E'_{k-2} \arrow[r,"\Phi'"] \arrow[u,hook] & (E'_{k+1}/E'_{k-1})K  \arrow[r] \arrow[u,hook] & L'_kK \arrow[u,hook]
		\end{tikzcd}
	\end{equation*} where the vertical maps are injections. By assumption the top sequence is the zero map, which implies that the bottom one is zero as well.
\end{proof}

\subsection{The rank $2$ case} \label{rk_2}

We start by describing the Lagrangian in rank $2$, where we are able to carry out some infinitesimal computations.

In the moduli of $\SL(2,\C)$-Higgs bundles, the only non-trivial fixed points of the $\C^*$-action are of type $(1,1)$, i.e., of the form 
\[\left( L \oplus L^{-1}, \begin{psmallmatrix}
	0 & 0 \\ 
	\gamma_0 & 0
\end{psmallmatrix} \right).\] Since $\gamma_0$ is a non-zero holomorphic section of $L^{-2}K$, it follows that $1 \le \deg L \le g-1$. 

Fixing the degree of $L$, the zero velocity Lagrangian will consist of points of the form
\[\left( M \oplus M^{-1}, \begin{psmallmatrix}
	0 & \beta \\ 
	\gamma & 0
\end{psmallmatrix} \right),\] with $\deg M = \deg L$, $\gamma \in H^0(M^{-2}K)$, $\beta \in H^0(M^2K)$.

We recognize these points as being $\SL(2,\R)$-Higgs bundles so that in fact, the Lagrangians we build this way are connected components of the moduli space $\mathcal M (\SL(2,\R))$ sitting inside $\mathcal M (\SL(2,\C))$, indexed by the degree of $L$ (see \cite{hitchinSDERS}). Since $\SL(2,\R)\cong\SU(1,1)$, this is in agreement with Remark \ref{SU_pq}.

\begin{rmk}\label{hitchin_2}
	If $\deg L = g-1$, then $\deg \gamma = 0$, so we obtain $2^{2g}$ Lagrangians $\mathcal L_\alpha$ which coincide with the Hitchin sections, as in Example \ref{hitchin_sec}. 
\end{rmk}


As a first step towards understanding the interaction of $\mathcal L_\alpha$ with the Hitchin map in \eqref{hitchin_map}, we compute its derivative.

\begin{prop}\label{rk2_der}
	Let $(E,\Phi) \in \mathcal L_\alpha$ be a point with smooth spectral curve. Then the derivative of the restriction $h|_{\mathcal L_\alpha}$ at the point $(E,\Phi)$ is an isomorphism.
\end{prop}

\begin{proof}
	We first need an appropriate description of the tangent space at a point of $\mathcal L_\alpha$. We know that such a point has the form \[(E,\Phi)=\left(L\oplus L^{-1},\begin{psmallmatrix}
		0 & \beta \\ 
		\gamma & 0
	\end{psmallmatrix}\right),\] with limit to $0$ equal to
	
	\[\mathcal E=\left(L\oplus L^{-1},\begin{psmallmatrix}
		0 & 0 \\ 
		\gamma & 0
	\end{psmallmatrix}\right).\] Note that although $(E,\Phi)$ may not be a fixed point, the bundle $E$ remains a direct sum so that the weight decomposition of $\End E$ given by \eqref{end_decomp} still holds.

From the deformation complex $C^\bullet(E,\Phi)$ in \eqref{def_comp} we notice that, unlike what happens at a fixed point (compare with \eqref{comp_decomp}), the map $[\Phi,-]$ sends $\End_0 E \cong \mathcal O$ into $\End_{-1}E \oplus \End_1 E \cong L^2K \oplus L^{-2}K$ so we have a subcomplex

\[
\begin{array}{rccc}
	C^\bullet_{\mathcal{L}}(E,\Phi): & \mathcal{O} & \longrightarrow & L^2 K \oplus L^{-2}K \\
	&\dot{\sigma} & \longmapsto & (-2\beta\dot{\sigma}, 2\gamma\dot{\sigma}).
\end{array}
\] Comparing with \eqref{slice_tangent}, we see that $T_{(E,\Phi)}\mathcal L_\alpha \cong \mathbb H^1(C^\bullet_{\mathcal{L}}(E,\Phi))$.

Now, it is known (see e.g. \cite[Proposition 8.2]{markmanSCIS}) that, at a point $(E,\Phi)$ of a smooth fiber of the Hitchin map $h$ (i.e., a point with smooth spectral curve), the kernel of the derivative $d_{(E,\Phi)}h$ is given by \[\ker d_{(E,\Phi)}h = \{[(\dot A,0)] \mid \dot A \in \Omega^{0,1}(\End E), [\Phi,\dot A]=0\}.\]
Thus, it follows that
\begin{align*}
 \ker d_{(E,\Phi)}(h|_\mathcal L) &= \ker d_{(E,\Phi)}h \cap \mathbb{H}^1(C^\bullet_{\mathcal L}(E,\Phi)) \\ 
&=\left\{\left[\begin{psmallmatrix}
\dot \sigma & 0 \\ 
0 & -\dot \sigma 
\end{psmallmatrix},0 \right] \,\middle| \, 0 = 2 \beta \dot\sigma, \, 0 = -2\gamma \dot \sigma\right\} \\
&=0.
\end{align*}

We conclude that, at a point $(E,\Phi) \in \mathcal L_\alpha$ with smooth spectral curve, the derivative $d_{(E,\Phi)}h|_\mathcal L$ is an isomorphism, since both $\mathcal L_\alpha$ and $\mathcal A$ have half the dimension of $\mathcal M$.
\end{proof}

Due to the properness of the Hitchin map, the surjectivity of the derivative allows us to show surjectivity of the map itself restricted to $\mathcal L_\alpha$. We note that in \cite[Proposition 22]{schap_sl2R} it was already shown that the intersection of a connected component of $\mathcal M(\SL(2,\R))$ (i.e., $\mathcal L_\alpha$) with any smooth fiber of $h$ is non-empty.

\begin{cor}
	The Hitchin map $h:\mathcal M \to \mathcal A = H^0(K^2)$ is surjective when restricted to the Lagrangian $\mathcal L_{\alpha}$.
\end{cor}

\begin{proof} 

	In this case, the Lagrangian $\mathcal L_\alpha$ is closed given that it was identified with a connected component of $\mathcal M(\SL(2,\R))$. Since the Hitchin map $h$ is proper, the image is also closed in the base $\mathcal A$. Moreover, it contains an open set of dimension equal to that of $\mathcal A$ by Proposition \ref{rk2_der} above, hence the image must be the whole of $\mathcal A$ (since $\mathcal A$, being a vector space, is irreducible).
\end{proof}

Since $h:\mathcal L_\alpha \to \mathcal A$ is surjective and generically a submersion, it follows that the intersection of $\mathcal L_\alpha$ with a generic fiber of $h$ is finite. Going through the proof of \cite[Proposition 22]{schap_sl2R}, we can actually find the number of intersection points.

\begin{prop} \label{thm:rk2_count}
	Let $a\in H^0(K^2)$ be such that $\div a$ is reduced. Let $F_\alpha$ be the component of fixed points of the form $\left(L\oplus L^{-1},\begin{psmallmatrix}
		0 & 0 \\ 
		\gamma & 0
	\end{psmallmatrix}\right)$ with $\deg L = d$. Then we have 
	\[|h^{-1}(a)\cap \mathcal L_\alpha| = 2^{2g} \binom{4g-4}{2g-2-2d}. \]
\end{prop}

\begin{proof}
	Following the proof of \cite[Proposition 22]{schap_sl2R}, consider a point in the Lagrangian of the form 
	\begin{equation}\label{rk2_lag}
			\left( L \oplus L^{-1}, \begin{psmallmatrix}
			0 & \beta \\ 
			\gamma & 0
		\end{psmallmatrix} \right).
	\end{equation} The Hitchin map in this case reduces to the determinant, which equals $-\beta \gamma$.

If we now fix $a \in H^0(K^2)$ with reduced divisor, it follows that for \eqref{rk2_lag} to be on $h^{-1}(a)$, it must necessarily satisfy 
\begin{equation}\label{divisor_2}
	\div \gamma \le \div a.
\end{equation}
 If $\deg L = d$, then $\deg \gamma = \deg L^{-2}K$ and so there are
\[\binom{\deg(K^2)}{\deg(L^{-2}K)}=\binom{4g-4}{2g-2-2d}\] choices for $\div \gamma$. After picking $\div \gamma$, there are $2^{2g}$ choices for $L$ and there are no further choices since $\beta = -a/\gamma$, thus we arrive at the number in the statement of the theorem.
\end{proof}

\begin{rmk}
	If $\deg L = g-1$, then $\deg \gamma = 0$ and we get $2^{2g}$ points of intersection with $h^{-1}(a)$. This coincides with Remark \ref{hitchin_2} since each Hitchin section intersects every fiber of the Hitchin map exactly once.
\end{rmk}

\subsection{$Z$-very stable and $Z$-wobbly fixed points} \label{z_nilpotents}

Given a fixed point component $F_\alpha \subset \mathcal M$, we have constructed a Lagrangian $\mathcal L_\alpha$ which fibers over it. For each fixed point $\mathcal E \in F_\alpha$, the fiber $\mathcal L_{\mathcal E}$ consists of all points that flow, as $t$ goes to $0$, to $\mathcal E$ with "zero velocity". One of the central themes of \cite{HH} is the classification of fixed points of the $\C^*$-action of type $(1,\ldots,1)$d into "very stable" or "wobbly", depending on if its upward flow contains any nilpotent elements other than the fixed point itself.

\begin{deff}
	A fixed point $\mathcal E$ of the $\C^*$-action is said to be \emph{very stable} if the intersection $W_{\mathcal E}^+ \cap h^{-1}(0)$ reduces to the fixed point $\mathcal E$ itself. It is said to be \emph{wobbly} otherwise.
\end{deff}

There, the very stable fixed points of type $(1,\ldots,1)$ are characterised via the following result (recall \eqref{divisors}).

\begin{theorem}{\cite[Theorem 4.16]{HH}} \label{HH_vs}
	Let $\mathcal E$ be a stable type $(1,\ldots,1)$ fixed point of the $\C^*$-action. Then $\mathcal E$ is very stable if and only if $\Phi_1 \Phi_2 \cdots \Phi_{n-1} \in H^0(L_1^*L_nK^{n-1})$ has no repeated zero, i.e., if and only if the effective divisor $\delta_1 + \cdots + \delta_{n-1}$ is reduced. 
\end{theorem}

To give an idea of the proof, given a fixed point $\mathcal E =(E,\Phi)$ and a non reduced point $c$ of $\delta^+$, we can construct a $1$-parameter family of $\Phi_c$-invariant subspaces $V_t$ of $E_c$ whose corresponding Hecke transforms define a curve inside the nilpotent cone connecting two distinct fixed point components, thus showing that $\mathcal E$ is wobbly. On the other hand, if all points of $\delta^+$ are reduced, we can show that the only possible $\Phi_c$-invariant subspaces of $E_c$, with $c \in \delta^+$, will define Hecke transformations that send the fixed point to another fixed point.

Given a fixed point $\mathcal E$, we can define analogues of very stable and wobbly considering the intersection of the nilpotent cone with the "zero velocity upward flow" $\mathcal L_{\mathcal E}$.

\begin{deff}
	A fixed point $\mathcal E$ is said to be \emph{$Z$-very stable} if the only nilpotent element of $\mathcal L_{\mathcal E}$ is the fixed point $\mathcal E$ itself. Otherwise we call it \emph{$Z$-wobbly}.
\end{deff}

A natural question that arises is then to see what happens to the nilpotent elements of $W_{\mathcal E}^+$ when passing to the zero velocity locus $\mathcal L_{\mathcal E}$. In particular, can a fixed point be wobbly but still $Z$-very stable? In other words, can the nilpotent elements be all outside the zero velocity locus? It is not hard to see that in rank $2$, this is indeed the case. In fact, a fixed point with chain form (recall \eqref{chain})
\[
\begin{tikzcd}
	L_1 \arrow[r,"\Phi_1"] & L_2,
\end{tikzcd}
\] is very stable if and only if $\Phi_1$ has no double zeros by Theorem \ref{HH_vs}, but since in this case the Lagrangian $\mathcal L_\alpha$ is a connected component of the moduli of $\SL(2,\R)$-Higgs bundles, its only nilpotent elements are the fixed points (it is also easy to see directly since the Higgs fields are of the form $\begin{psmallmatrix}
	0 & \beta \\ 
	\gamma & 0
\end{psmallmatrix}$, with $\gamma \ne 0$). Thus, although there exist wobbly fixed points, there are no $Z$-wobbly fixed points in rank $2$, in other words, all wobbly fixed points become $Z$-very stable (and of course the very stable ones remain $Z$-very stable).

This next result characterizes the $Z$-very stable $\C^*$-fixed points of type $(1,\ldots,1)$ in any rank through their associated divisors $(\delta_1,\cdots,\delta_{n-1}$). The point is that although any fixed point with non-reduced divisor will be wobbly, if these multiplicities are properly arranged (in this case, "concentrated at the extremities"), then the nilpotent orbits will lie outside the zero velocity locus and the point will be $Z$-very stable. We essentially adapt the proof of Theorem \ref{HH_vs}. The main novelty is to determine which of the families of invariant subspaces $V_t$ will preserve the zero velocity condition when a Hecke transformation is performed.

\begin{theorem}
	\ThmC
	\label{thm:C}
\end{theorem}

\begin{proof}

	We first show that if the divisor is as in the statement of the theorem, then the fixed point must be $Z$-very stable. We go by induction on $\deg \delta_+$. If it is zero, then $\mathcal L_{\mathcal E}=W_{\mathcal E}^+$ is the Hitchin section as noted in Example \ref{hitchin_sec}, hence the result is clear (the point $\mathcal E$ is in fact very stable). Otherwise, let $c \in \delta_+$, let $k$ be such that $\Phi_k(c)=0$ and consider a nilpotent $(E,\Phi) \in \mathcal L_{\mathcal E}$. The goal is to show that this nilpotent element must in fact be isomorphic to the fixed point itself. Letting $(E_i)_i$ be the Simpson filtration of $(E,\Phi)$, we perform a Hecke transformation at the $(\Phi_c)$-invariant subspace $(E_k)_c$, obtaining a nilpotent Higgs bundle $(E',\Phi')$ flowing to another fixed point $\mathcal E'$. If $\mathcal E$ has chain form
	\begin{equation*}
		\begin{tikzcd}
			L_1 \arrow[r,"\Phi_1"] & L_2 \arrow[r,"\Phi_2"] & \cdots \arrow[r,"\Phi_{n-1}"] & L_{n}, 
		\end{tikzcd}
	\end{equation*} then $\mathcal E' = (\bigoplus_i L'_i,\overline{\Phi'})$ will be of the form
	\begin{equation*}
		\begin{tikzcd}
			L_1 \arrow[r,"\Phi_1"] & \cdots \arrow[r,"\Phi_{k-1}"] & L_k \arrow[r,"\Phi_{k}/s_c"] & L_{k+1}(-c) \arrow[r,"\Phi_{k+1}"] & \cdots \arrow[r,"\Phi_{n-1}"] & L_n(-c).
		\end{tikzcd}
	\end{equation*} Since the divisor $\delta_+'$ associated to $\mathcal E'$ equals $\delta_+ - c$, by induction $\mathcal E'$ is a $Z$-very stable fixed point. If we check that $(E',\Phi')$ is in the zero velocity locus $\mathcal L_{\mathcal E'}$ then we have $(E',\Phi')\cong \mathcal E'$. As in the proof of Proposition \ref{zv_hecke}, we need to see that, for each $i$, we have commutative exact diagrams of the form
	
	\[
	\begin{tikzcd}
		&0\arrow[d] & 0\arrow[d] & 0\arrow[d] \\ 
		0\arrow[r]& E'_{i-2} \arrow[r] \arrow[d] & E_i' \arrow[r] \arrow[d] & L'_{i-1} \oplus L'_i \arrow[d]\arrow[r] & 0\\ 
		0\arrow[r]&E_{i-2}\arrow[r] \arrow[d] & E_i \arrow[r,"p"] \arrow[d] & L_{i-1}\oplus L_i \arrow[d]\arrow[r] & 0\\ 
		0\arrow[r] & (E_{i-2})_c/V_{i-2} \arrow[r] \arrow[d] & (E_i)_c/V_i \arrow[r] \arrow[d] & (L_{i-1}\oplus L_i)_c/p_c(V_i) \arrow[d]\arrow[r] & 0 \\ 
		& 0 & 0 & 0, 
	\end{tikzcd}
	\] with $V_i = (E_i)_c \cap (E_k)_c = (E_{\min(i,k)})_c$. This is equivalent to having
	
	\[	
	p_c(V_i) = \begin{cases}
		(L_{i-1}\oplus L_i)_c \text{ if } i \le k \\ 
		(L_k)_c \text{ if } i=k+1 \\ 
		0_c \text{ if } i \ge k+2
	\end{cases}
	\] which is easily checked to be true.
	
	Now if the order of $c$ in $\delta_+$ is 1, then it becomes 0 in $\delta_+'$, so $\Phi'_c$ is regular nilpotent and the argument runs as in the proof of Theorem \ref{HH_vs}---there is a unique $(n-k)$-dimensional $\Phi'_c$-invariant subspace $V'$ of $E'_c$, namely $V' = (L_{k+1}\oplus \cdots \oplus L_n)_c$, and the Hecke transform at this subspace of $(E',\Phi')\cong \mathcal E'$ return $(E(-c),\Phi)\cong \mathcal E(-c)$ which shows the result after twisting by $\mathcal O(c)$.
	
	If the order of $c$ in $\delta_+$ is strictly bigger then $1$ then by assumption, $k$ is either $1$ or $n-1$. We claim that, as in the proof of Theorem \ref{HH_vs}, although we are able to find curves of nilpotent elements in $W_{\mathcal E}^+$, showing that the fixed point $\mathcal E$ is wobbly, all of these are in fact outside of the zero velocity locus $\mathcal L_{\mathcal E}$ and so the fixed point remains $Z$-very stable.
	
	Assume first $k=1$, then the Higgs field $\Phi'_c$ is of the form
	\[		
	\begin{pmatrix}
		0 & & & & \\
		0 & 0 \\
		& 1 & 0 \\ 
		& & \ddots & \ddots \\ 
		& & & 1 & 0
	\end{pmatrix}
	\] and we seek $(n-1)$-dimensional invariant subspaces of $E'_c \cong (L'_1 \oplus \cdots L'_n)_c$. Letting $(e_i)$ be a basis vector of $(L'_i)_c$, we can represent any $n-1$ dimensional subspace $V'$ by a $(n-1)\times n$ matrix where the rows are the coordinates of a basis of $V$ with respect to the $e_i$. Performing a Hecke at such a $V'$ will give a new point $(E'',\Phi'')$ which will flow to the fixed point $\mathcal E''$ given in chain form by
	
	\begin{equation*}
		\begin{tikzcd}
			L'_1(-c) \arrow[r,"\Phi_1"] & L'_2 \arrow[r,"\Phi_{2}"] & \cdots \arrow[r,"\Phi_{n-1}"] & L'_{n}. 
		\end{tikzcd}
	\end{equation*}
	
	Now the Hecke transform at $V'$ will preserve the zero velocity condition if and only if
	
	\[	
	p_c(V_i) = \begin{cases}
		(L_2)_c \text{ if } i=2 \\ 
		(L_{i-1}\oplus L_i)_c \text{ if } i>2. \\
	\end{cases}
	\] Using these conditions together with Gauss elimination, we can represent such a $V'$ as a matrix of the form (where the rows are a basis for $V'$)
	
	\[	
	\begin{pmatrix}
		0 & 1 & 0\\ 
		a_1 & 0 & 1 & \ddots\\ 
		\vdots &  & \ddots & \ddots & 0\\ 
		a_{n-2} & & & 0 & 1
	\end{pmatrix}.
	\] In other words we can pick a basis of $V'$ of the form
	\[e_2, a_1e_1+e_3,\ldots,a_{n-2}e_1+e_n.\] Now imposing invariance of $V'$ by $\Phi'_c$ it is easily checked that in fact all $a_j$ must be equal to zero. Thus, as in the regular nilpotent case, there is a unique $(n-1)$-dimensional subspace whose Hecke transformation preserves the zero velocity condition.
	
	The situation if $k=n-1$ is similar. Now we seek a line $V'$ which can be represented as a single vector that will preserve the zero velocity condition if and only if it is of the form
	\[ \begin{pmatrix}
		a_1 & \cdots & a_{n-2} & 0 & 1
	\end{pmatrix}, \] i.e., it is generated by \[a_1e_1 + \cdots a_{n-2}e_{n-2}+e_n.\]
	
	The line needs to be invariant by $\Phi'_c$, given by the matrix  
	\[		
	\begin{pmatrix}
		0 & & & & \\
		1 & 0 \\
		& 1 & 0 \\ 
		& & \ddots & \ddots \\ 
		& & & 0 & 0
	\end{pmatrix}
	\] so again all $a_j$ must be equal to zero, which means the desired $V'$ is also unique. 
	
	For the reverse implication, we will take a divisor which is not of the form in the statement of the theorem and from it produce a Hecke curve of nilpotent Higgs bundles which lie in the zero velocity locus. A divisor can fail to be as in the statement in two different ways. Either there is a common zero between two different $\Phi_i$ or one of the non-extremal $\Phi_i$ has a multiple zero.
	
	We first take a fixed point $\mathcal E = (\bigoplus_i L_i,\Phi)$ with $i < j$ such that $\Phi_i(c)=\Phi_j(c)=0$, for some $c \in X$. As in the first part of the proof, we perform a Hecke transformation at $(L_1 \oplus \cdots \oplus L_i)_c$, deleting the zero at $\Phi_i$ and obtaining $\mathcal E' = (\bigoplus_i L'_i,\Phi')$. If $\Phi'_k(c)\ne 0$ for $k \ne j$, the resulting Higgs field $\Phi'_c$ can be described by the chains
	\[e_1 \rightarrow e_2 \rightarrow \cdots \rightarrow e_i \rightarrow \cdots \rightarrow e_j \rightarrow 0,\] 
	\[e_{j+1} \rightarrow e_{j+2} \rightarrow \cdots \rightarrow e_n \rightarrow 0,\] where as before $e_k$ is a basis vector of $(L'_k)_c$. Consider the 1-parameter family of $(n-i)$-dimensional subspaces of $\mathcal E'$ given by
	\[V_t = \langle e_{i+1},\ldots,e_j \rangle \oplus \langle e_{j+1}+te_i,\ldots,e_n \rangle.\] It is $\Phi'_c$-invariant, it satisfies \[V_0 = \langle e_{i+1},\ldots,e_n\rangle \ne V_\infty\] and the Hecke transform at $V_t$ of $\mathcal E$ will preserve the zero velocity condition (recall that it is trivially satisfied by any fixed point) since the non-trivial extension happens between $L_i$ and $L_{j+1}$ and $j+1-i > 1$. Notice as well that if there is some other $k \ne j$ such that $\Phi'_k(c)=0$, each $V_t$ is still $\Phi'_c$-invariant. Thus $\mathcal H_{V_t}(\mathcal E')$ is a Hecke curve of nilpotent Higgs bundles flowing to $\mathcal H_{V_0}(\mathcal E')=\mathcal E(-c)$ with velocity zero. So $\mathcal E(-c)$ (and thus $\mathcal E$) is $Z$-wobbly.
	
	Now to the second case, assume that for some $1<i<n-1$, $\Phi_i$ has a zero $c$ with multiplicity at least $2$. After performing the first Hecke ("deleting the zero") to obtain $\mathcal E'$, we can assume that for all $k \ne i$, we have $\Phi_k(c)\ne 0$ (otherwise we already know that the point is $Z$-wobbly), so that $\Phi'_c$ can be described again by two chains
	\[e_1 \rightarrow e_2 \rightarrow \cdots \rightarrow e_i \rightarrow 0\] 
	\[e_{i+1} \rightarrow e_{j+2} \rightarrow \cdots \rightarrow e_n \rightarrow 0.\]
	If $i \le n-i$, we can take the family of subspaces
	\[V_t = \langle e_{i+1}+te_1,\ldots,e_{2i}+te_i \rangle \oplus \langle e_{2i+1},\ldots,e_n \rangle\] and if $i \ge n-i$, we can take
	\[V_t = \langle e_{i+1}+te_{2i-n+1},\ldots,e_{n}+te_i \rangle.\]
	
	In both cases, we again get a family of subspaces of $\mathcal E'_c$ with the desired properties, thus producing a Hecke curve of zero velocity nilpotents whose limit is $\mathcal E(-c)$, so it is $Z$-wobbly.
\end{proof}

\section{Intersection with a generic fiber} \label{int_generic}

Analyzing the proof of Proposition \ref{thm:rk2_count}, we can identify three main steps:

\begin{enumerate}
	\item we first counted the number of possible fixed points $\mathcal E \in F_\alpha$ for which the intersection $\mathcal L_{\mathcal E} \cap h^{-1}(a)$ can be non-empty. This led us to the condition $\div \gamma \le \div a_2$ on the fixed point $\mathcal E$;
	\item given such a fixed point, we counted the number of possible intersections in $\mathcal L_{\mathcal E} \cap h^{-1}(a)$ and obtain only one since $\beta$ is determined by $a$ and $\gamma$;
	\item finally, we showed that for each choice made in the previous two points, there actually exists a corresponding intersection point---this was easy because we just set $\beta = -a_2/\gamma$.
\end{enumerate}

In this section, we will show a formula for the number of intersection points between a Lagrangian $\mathcal L_\alpha$ built on a component $F_\alpha$ of type $(1,\ldots,1)$ fixed points of the form \eqref{chain} and a generic fiber $h^{-1}(a)$, by generalizing each step in the following ways:
\begin{enumerate}
	\item given a point $a \in \mathcal A$, we will find divisors $D_{n,i}$ on the curve $X$ such that $\div \Phi_{i} \subset D_{n,i}$, which will enable us to determine the finite set of fixed points in $F_\alpha$ such that $\mathcal L_{\mathcal E} \cap h^{-1}(a)$ can be non-empty;
	\item given such a fixed point, we will use a result of \cite{HH} to count the number of possible intersection points by relating them to the eigenvalues of $\Phi$ at the zeros of each $\Phi_i$;
	\item given a choice of fixed point as in $1.$ and a choice of eigenvalues as in $2.$, we will produce an intersection point.
\end{enumerate}

Observing that, for any fixed point $\mathcal E$ and point $a$ in the Hitchin base, we have $\mathcal L_{\mathcal E} \cap h^{-1}(a) \subset W^+_{\mathcal E}  \cap h^{-1}(a)$, our main technical tool will be the result of \cite{HH} which counts the number of points in $W^+_{\mathcal E}  \cap h^{-1}(a)$. We reproduce this result below for the convenience of the reader (we have slightly changed its phrasing so it makes more sense outside the context of the paper).

\begin{prop}{\cite[Proposition 5.18]{HH}} \label{HH_count}
	Let $\mathcal E \in \mathcal M$ be a very stable type $(1,\ldots,1)$ $\C^*$-fixed point with chain form \eqref{chain}. Let $a \in \mathcal A$ be such that the spectral curve $X_a \subset |K|$ is smooth and $\delta_+$ (recall \eqref{delta_plus}) avoids the ramification divisor of $\pi_a:X_a \longrightarrow X$.
	
	\begin{enumerate}
		\item 	To each Higgs bundle $(E,\Phi) \in W_{\mathcal E}^+ \cap h^{-1}(a)$ corresponds a choice of reduced effective divisors $\Gamma_i$ supported on $(n-i)$ preimages in $X_a$ of $\pi_a$ over each zero of $\Phi_i$ in $X$. More precisely, if $c \in X$ is a zero of $\Phi_i$ then $\Phi_c$ leaves $(E_i)_c$ invariant and $\tilde c \in \pi_a^{-1}(c)$ is in the support of $\Gamma_i$ if and only if $\tilde c \in K_c$ is \textbf{not} an eigenvalue for the restriction $\Phi_c|_{(E_i)_c}: (E_i)_c \to (E_i)_cK_c$ (here $E_\bullet$ denotes the Simpson filtration).
		
		\item Conversely, to each choice of reduced divisor $\Gamma_i$ supported on an $(n-i)$ element subset of $\pi^{-1}_a(c)$ for all zeros $c$ of all $\Phi_i$, there is a unique $(E,\Phi) \in W_{\mathcal E}^+ \cap h^{-1}(a)$ such that the corresponding line bundle on $X_a$ (given by the BNR correspondence) has the form $U = \pi_a^*(E_1)(\sum_i \Gamma_i)$.
		
		\item The above two correspondences are inverse of each other. In particular, counting the number of possible divisors $\Gamma_i$, we obtain
		\[|W^+_{\mathcal E}  \cap h^{-1}(a)|=\prod_{i=1}^{n-1}\binom{n}{i}^{m_i},\] where $m_i = \deg(\Phi_i)$.
	\end{enumerate}

\end{prop}

It will be useful for us to be able to distinguish those fixed points which are limits (to $0$) of the intersection points of the Lagrangian $\mathcal L_\alpha$ and a fiber of the Hitchin map.

\begin{deff}\label{admissible}
	Let $F_\alpha$ be a $\C^*$-fixed point component and $a \in \mathcal A$. We say that a fixed point $\mathcal E \in \mathcal L_\alpha$ is \emph{$a$-admissible} if $\mathcal L_{\mathcal E} \cap h^{-1}(a) \ne \varnothing$. We denote the set of $a$-admissible fixed points of $F_\alpha$ by $F_\alpha(a)$. 
\end{deff}


We start by treating the cases of rank $3$ and $4$ separately since these turn out to be particularly simple to understand and already contain the main ideas of the general case. However we save part $3$ of the above program to the general case since there is nothing to be gained by restricting to lower rank.

\begin{rmk} 
	In rank $2$, recall that a point of the Lagrangian $\mathcal L_\alpha$ is of the form \[\left( L \oplus L^{-1}, \begin{psmallmatrix}
		0 & \beta \\ 
		\gamma & 0
	\end{psmallmatrix} \right).\] Thus, the ramification divisor of the corresponding spectral curve, defined by $\lambda^2 - \gamma\beta$, contains the divisor of $\gamma$. This excludes the rank 2 case from fitting in Proposition \ref{HH_count} since there it is assumed that the ramification divisor and the divisors of the $\Phi_i$ are disjoint. It is however interesting to note that, from Proposition \ref{thm:rk2_count}, given an admissible fixed point $\mathcal E$, there is a single intersection point in $h^{-1}(a) \cap \mathcal L_{\mathcal E}$ and this agrees with the fact that over each zero of $\gamma$ there is now only a single pre-image of the spectral curve. This suggests that it may be possible to extend Proposition \ref{HH_count} to the ramified locus. 
\end{rmk}

\begin{prop} \label{prop:eigenvalues}
	Let $F_\alpha$ be a $\C^*$-fixed point component of type $(1,1,1)$ with $m_i = \deg \Phi_i$ and let $a=(a_2,a_3) \in \mathcal A$ be a point on the Hitchin base such that $a_3$ has simple zeros disjoint from the ramification divisor of the spectral curve $X_a$. Then we have \[|\mathcal L_\alpha \cap h^{-1}(a)|\le 3^{2g}\binom{6g-6}{m_1,m_2}.\]
\end{prop}

\begin{proof}
A point in $\mathcal L_\alpha \cap h^{-1}(a)$ has a Higgs field of the form
\begin{equation}\label{field_3}
	\Phi = \begin{pmatrix} 
		0 & \phi_1 & \phi_3 \\ 
		\Phi_1 & 0 & \phi_2 \\ 
		0 & \Phi_2 & 0
	\end{pmatrix}
\end{equation}
so it satisfies
\begin{equation}\label{divisor_3}
	\div \Phi_1 \le \div a_3, \qquad \div \Phi_2 \le \div a_3.
\end{equation}

  Since we are fixing $a_3 \in H^0(K^3)$, there are
\[\binom{6g-6}{m_1,m_2}\] ways to distribute its zeros among $\Phi_1$ and $\Phi_2$ and then there are $3^{2g}$ choices for the fixed points once $\div \Phi_1$ and $\div \Phi_2$ are chosen. 
Note that since $a_3$ has reduced divisor, all these choices give rise to very stable fixed points.

Given $\mathcal E \in F_\alpha(a)$, a point in the intersection $\mathcal L_\mathcal E \cap h^{-1}(a)$ is in particular a point in $W_{\mathcal E}^+ \cap h^{-1}(a)$, as remarked earlier. Since $\mathcal E$ is very stable, it corresponds, by Proposition \ref{HH_count}, to a choice of two eigenvalues of $\Phi_c$ over every zero $c$ of $\Phi_1$ and a choice of one eigenvalue of $\Phi_c$ over every zero $c$ of $\Phi_2$.


	
	Looking at \eqref{field_3}, we see that when $\Phi_1(c) = 0$, the map $\Phi|_{E_1}$ is zero at $c$, thus its only eigenvalue is zero and we must choose the other two non-zero points of the spectral curve. When $\Phi_2(c)=0$, the map $(\Phi|_{E_2})_c$ has two non-zero symmetric eigenvalues so the one we must choose is the zero of the spectral curve.

	\begin{figure}[htbp]
		\centering
		\includegraphics[width=0.5\textwidth]{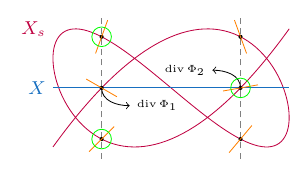}
		\caption{}
		\label{fig:my_image}
	\end{figure}

	Thus, out of all possible choices of pre-images of the spectral cover over the zeros of $\Phi_1$ and $\Phi_2$, there is in fact only one which can arise from a point in the Lagrangian: we must necessarily choose the non-zero eigenvalues over the zeros of $\Phi_1$ and the zero eigenvalue over the zeros of $\Phi_2$. This shows that there can exist at most one point in the intersection $|\mathcal L_{\mathcal E} \cap h^{-1}(a)|$, for each $\mathcal E \in F_\alpha(a)$.
	\end{proof}

If $n$ is even, the eigenvalue choices will double at the zeros of $\Phi_{n/2}$, as can be seen in the case of rank 4.

\begin{prop}
	Let $F_\alpha$ be a fixed point component of type $(1,1,1,1)$ with $\deg \Phi_i = m_i$, for $i=1,2,3$.

	Then, for $a=(a_2,a_3,a_4) \in \mathcal A$ with $\div(a_3a_4)$ reduced and such that $\div (\Phi_1 \Phi_2 \Phi_3)$ is disjoint from the ramification divisor of $X_a$, we have
	\[|\mathcal L_\alpha \cap h^{-1}(a)| \le 4^{2g}2^{m_2}\binom{8g-8}{m_1,m_3}\binom{6g-6}{m_2}.\] 
\end{prop}

\begin{proof}
	Let $\mathcal E \in F_\alpha$ be a type $(1,\ldots,1)$ fixed point given in chain form by
	\begin{equation}\label{fixed_4}
		\begin{tikzcd}
			L_1 \arrow[r,"\Phi_1"] & L_2 \arrow[r,"\Phi_2"] & L_3 \arrow[r,"\Phi_3"] & L_4
		\end{tikzcd}
	\end{equation}
	with $\deg \Phi_i = m_i$. The Higgs field of a point $(E,\Phi) \in \mathcal L_{\mathcal E}$ takes the form
	\begin{equation}\label{field_4}
		\Phi = \begin{pmatrix}
			0 & \phi_1 & \phi_2 & \phi_3 \\ 
			\Phi_1 & 0 & \phi_4 & \phi_5 \\ 
			0 & \Phi_2 & 0 & \phi_6 \\ 
			0 & 0 & \Phi_3 & 0 
		\end{pmatrix}.
	\end{equation}
	 Calculating its characteristic polynomial, we obtain
	\[ \lambda^4 - (\Phi_1 \phi_1 + \Phi_2 \phi_4 + \Phi_3 \phi_6)\lambda^2 - \Phi_2(\Phi_1 \phi_2 + \Phi_3\phi_5)\lambda + \Phi_1 \Phi_3 (-\Phi_2 \phi_3 + \phi_1 \phi_6) \] so, if $(E,\Phi) \in h^{-1}(a_2,a_3,a_4)$, it follows that 
\begin{equation}\label{divisor_4}
		\div \Phi_1 \le \div a_4, \qquad \div \Phi_2 \le \div a_3,\qquad \div \Phi_3 \le \div a_4.
\end{equation}

	Thus, fixing $a \in \mathcal A$ with $\div(a_3a_4)$ reduced (this condition will be generalised for any rank in Proposition \ref{prop:generic}), these inequalities impose conditions on the fixed point $\mathcal E$ given by \eqref{fixed_4}, which imply that there are at most
	\begin{equation}
		4^{2g}\binom{8g-8}{m_1}\binom{8g-8-m_1}{m_3}\binom{6g-6}{m_2}
	\end{equation}
	 fixed points in $F_\alpha(a)$ which will again be all very stable, so Proposition \ref{HH_count} applies.
	

	If $c \in \div \Phi_1\Phi_3$, looking at \eqref{field_4}, the situation is very similar to the rank $3$ case since $0$ will again be an eigenvalue and so we either choose the non-zero eigenvalues (if  $\Phi_1(c)=0$)  or $0$ (if $\Phi_{3}(c)=0$), giving rise to just one choice in any case.


	If $c \in \div \Phi_2$, the Higgs field of a point in $\mathcal L_\alpha$ takes the form
\[\Phi_c = \begin{pmatrix}
		0 & \phi_1 & \phi_2 & \phi_3 \\ 
		\Phi_1 & 0 & \phi_4 & \phi_5 \\ 
		0 & 0 & 0 & \phi_6 \\ 
		0 & 0 & \Phi_3 & 0 
	\end{pmatrix}.\] Thus $(E_2)_c$ is $\Phi_c$-invariant and the trace of its restriction is zero. Since $\Phi_c$ itself also has zero trace, this means that the eigenvalues of $\Phi_c$ can be partitioned into two blocks  $\{x,-x\}$ and $\{y,-y\}$ with $x \ne y$ and both non-zero (since $c$ is not in the ramification divisor of $X_a$, by hypothesis). Working backwards, after choosing $\div \Phi_2$ from the zeros of $a_3$, we can choose either one of the two pairs of eigenvalues that sum to $0$, giving rise to $2$ choices for each zero of $\Phi_2$.



	We then conclude that, given an admissible fixed point, there are at most
	$1^{m_1}2^{m_2}1^{m_3}$ choices of eigenvalues, i.e., we have $|\mathcal L_{\mathcal E} \cap h^{-1}(a)| \le 2^{m_2}$, for each $\mathcal E \in F_\alpha(a)$. \qedhere 

\end{proof}


	For rank $n\ge 5$, it is not true in general that the $\Phi_i$ divide some coefficient $a_k$ of the characteristic polynomial of $\Phi$, as in \eqref{divisor_2}, \eqref{divisor_3} and \eqref{divisor_4}, although it is easy to see that $\Phi_1$ and $\Phi_{n-1}$ always divide the determinant $a_n$. This was fundamental in the above cases to obtain a bound on the number of possible divisors of the $\Phi_i$. What can be observed is that, at a zero $c$ of $\Phi_i$, the restriction $\Phi_c$ to $(E_i)_c$ has trace zero so the characteristic polynomial of $\Phi_c$ must split into a product of two polynomials associated to traceless matrices. More precisely, there must exist some $b_k,c_k \in K^k_c \cong \C$ such that the following equation holds
\begin{equation} \label{eq:traceless}
	\lambda^n + a_2(c) \lambda^{n-2} + \cdots + a_n(c) = (\lambda^i + b_2 \lambda^{i-2}+\cdots +b_i)(\lambda^{n-i} + c_2 \lambda^{n-i-2}+\cdots +c_{n-i}).
\end{equation}
Such a decomposition exists if and only if the coefficientes of the characteristic polynomial of $\Phi_c$ satisfy a certain polynomial relation, as shown in the following result.

\begin{prop}\label{resultant}
	For each $n \ge 3$ and $1\le i \le n-1$, there exists an irreducible polinomial
	\[R_{n,i} \in \C[X_2,\ldots,X_n]\] that vanishes at a point $(a_2(c),\ldots,a_n(c)) \in \C^{n-1}$ if and only if a decomposition of type \eqref{eq:traceless} exists.
\end{prop}

\begin{proof}
	Let $Z_k$ denote the affine space of monic degree $k$ polynomials whose degree $k-1$ coefficient is zero. We are interested in the image of the following algebraic map
	\[\begin{array}{cccc}
		\phi: & Z_i \times Z_{n-i} & \longrightarrow & Z_n \\ 
		& (f,g) & \longmapsto & f\cdot g.
	\end{array}\] Note that a decomposition of the type \eqref{eq:traceless}, when it exists, is determined by a partition of the $n$ roots of the polynomial on the left hand side in two blocks of size $i$ and $n-i$, hence there at most $\binom{n}{i}$ preimages of $\phi$ over a point in $Z_n$, in particular, they are all zero-dimensional. Since $\dim(Z_i \times Z_{n-i}) = (i-1)+(n-i-1) = n-2$ and $\dim Z_n = n-1$, we find that the Zariski closure of the image has codimension $1$ in $Z_n$. Thus there exists a polynomial $R_{n,i} \in \C[X_2,\ldots,X_n]$, the coordinate ring of $Z_n$, such that $\overline{\im \phi}$ is precisely the vanishing locus of $R_{n,i}$. Moreover, since the domain of $\phi$ is irreducible, it follows that $R_{n,i}$ is an irreducible polynomial.
	
	It remains to show that the image of $\phi$ is in fact closed in the Zariski topology. It is a constructible subset of an affine variety, so it is Zariski-closed if and only if it is closed in the complex analytic topology. Thus let $(P_k)_k$ be a sequence of polynomials in the image of $\phi$ converging to $P \in Z_n$. Denote the ordered set of roots of each $P_k$ by $(\lambda_k^j)_{j=1,\ldots,n}$, where we can assume that the first $i$ roots sum to zero. Since $P_k$ converges, the set of all roots is bounded, so taking a subsequence of $\lambda_k^j$, for each $j$, we can assume they converge to a set $\lambda_\infty^j$ of points of $\C$. Clearly $\sum_{j=1}^i \lambda_\infty^j = 0$ and
	\[P(\lambda) = \lim_k P_k(\lambda) = \lim_k (\lambda-\lambda_k^1)\cdots (\lambda-\lambda_k^n)=(\lambda-\lambda_\infty^1)\cdots (\lambda-\lambda_\infty^n).\] This means that the $\lambda_\infty^j$ are roots of $P$, so $P$ has a factorization of type \eqref{eq:traceless}, i.e., it belongs to the image of $\phi$.
\end{proof}

The following result shows that it makes sense to evaluate each polynomial $R_{n,i}$ at a point of the Hitchin base.

\begin{lemma} \label{lem:homo}
	The polynomial $R_{n,i}(X_2,\ldots,X_n)$ is weighted homogeneous, i.e., for all $t \in \C$, 
	\[R_{n,i}(t^2 X_2,\ldots,t^n X_n) = t^{d(n,i)} R_{n,i}(X_2,\ldots,X_n),\] for some $d(n,i) \in \mathcal \N$, which we call the \emph{weighted degree} of $R_{n,i}$.

	In particular, for every $a=(a_2,\ldots,a_n) \in \mathcal A$, the expression  
	\[R_{n,i}(a) \coloneqq R_{n,i}(a_2,\ldots,a_n)\]
	represents a well-defined holomorphic section of $K^{d(n,i)}$.
	
\end{lemma}
\begin{proof}
	Let $c \in X$ and assume that the coefficients $a_k(c)$ are such that the decomposition \eqref{eq:traceless} holds. Equating coefficients on both sides, we have equations of the form
	\[a_k(c) = \sum_j b_j c_{k-j}\] in the variables $b_j,c_j$. If we multiply $a_k(c)$ by $t^k$ then we obtain true equations if we multiply each $b_j$ and $c_j$ by $t^j$.
	Thus, by Proposition \ref{resultant}, we have, for all $t \in \C$ and $(x_2,\ldots,x_n) \in \C^{n-1}$, \[R_{n,i}(x_2,\ldots,x_n)=0 \implies R_{n,i}(t^2x_2,\ldots,t^nx_n)=0. \] Now consider the polynomial $P_{n,i}(Y_2,\ldots,Y_n)\coloneqq R_{n,i}(Y_2^2,\ldots,Y_n^n)$. It satisfies
	\[P_{n,i}(y_2,\ldots,y_n)=0 \implies P_{n,i}(ty_2,\ldots,ty_n)=0,\] for every $t \in \C$ and $(y_2,\ldots,y_n) \in \C^{n-1}$. Thus the zero set of $P_{n,i}$ is the cone of a projective variety, meaning that $P_{n,i}$ is homogeneous, i.e., 
	\[P_{n,i}(tY_2,\ldots,tY_n) = t^d P_{n,i}(Y_2,\ldots,Y_n)\] for some natural number $d$. Replacing $Y_k = X^{k}$, this implies that $R_{n,i}$ is weighted homogeneous.
\end{proof}

Thus the fact that at each zero of $\Phi_i$ we have a decomposition of the type \eqref{eq:traceless} can be translated as the inequality 
\begin{equation}\label{ineq:divisor}
	\div\Phi_i \le \div R_{n,i}(a).
\end{equation}
Here are some easy observations:
\begin{itemize}
 	\item the resultant $R_{n,i}$ does not change if we switch the variables $b_k$ with $c_k$, so
 	\begin{equation}\label{R_symmetry}
 		R_{n,i}=R_{n,n-i};
 	\end{equation} 
 	
 	\item if $i=1$, equation \eqref{eq:traceless} has a (unique) solution at $c \in X$ if and only if $a_n(c)=0$, so $R_{n,1} = R_{n,n-1}=X_n \in \C[X_2,\ldots,X_n]$. 
 \end{itemize}

The following result computes the weighted degress $d(n,i)$.

\begin{prop}\label{R_degree}
	For rank $n \ge 3$, the polynomials $R_{n,i}(X_2,\ldots,X_n)$ have weighted degree
\[d(n,i) = \begin{cases}
			\binom{n}{i} \text{ if } i \neq n/2 \\ 
			\frac{1}{2}\binom{n}{i} = \binom{n-1}{i-1} \text{ if } i = n/2.
		\end{cases}\]
 \end{prop}

 \begin{proof}
 	Consider the ring of polynomials $\C[\lambda_1, \ldots, \lambda_n]/I$ where $I$ is the ideal generated by $\sum_j \lambda_j$.
 	
 	
 	For each $n\ge 3$ and $J \subseteq \{1,\ldots,n\}$, set
    \begin{equation}\label{S_J}
        S_J(\lambda_1,\ldots,\lambda_n)\coloneqq \sum_{j \in J}\lambda_j.
    \end{equation}
    For each $1\le i \le n-1$, consider the polynomial 
\begin{equation}\label{S_ni}
    S_{n,i}(\lambda_1,\ldots,\lambda_n) = \prod_{\substack{J \subseteq \{1,\ldots,n\} \\ |J| = i }} S_J(\lambda_1,\ldots,\lambda_n).
\end{equation}
    which is clearly symmetric in the variables $\lambda_1,\ldots,\lambda_n$. Hence there exists a polynomial $\widetilde S_{n,i} \in \C[X_2,\ldots,X_n]$ such that
 	\[S_{n,i}(\lambda_1,\ldots,\lambda_n) = \widetilde S_{n,i}(e_2,\ldots,e_n),\] where the $e_j$ are the elementary symmetric polynomials in the variables $\lambda_j$. Multiplying each $\lambda_j$ by a scalar $t$ results in $S_{n,i}$ being multiplied by $t^{\binom{n}{i}}$  and the same is true for $\widetilde S_{n,i}$ when multiplying each $e_j$ by $t^j$.

Note that, at a point $c \in X$ such that equation \eqref{eq:traceless} holds, the roots of the characteristic polynomial can be partitioned into two subsets, of size $i$ and $n-i$, such that the sum of the elements in each subset sums to $0$. By the definition of $R_{n,i}$, this means that the vanishing sets of $\widetilde S_{n,i}$ and $R_{n,i}$ are the same. Thus, 
 $\widetilde S_{n,i} = R_{n,i}^m$, for some $m \ge 1$.

 We claim that
\[m = \begin{cases}
	 1 \text{ if } i \neq n/2 \\ 
	 2 \text{ if } i = n/2,
\end{cases}\] which will show the intended result.

If $i \ne n/2$, we can assume $i<n/2$ by \eqref{R_symmetry}. We claim that, for any subset $J$ of size $i$, $S_{n,i}$ would have $S_J^{m}$ as a factor. Indeed, let $S_J$ be a factor of $S_{n,i}$, it will be a factor of $R_{n,i}(e_2,\ldots,e_n)^m$. Thus $S_J$, being prime by Lemma \ref{prime} below, is a factor of $R_{n,i}(e_1,\ldots,e_n)$, which shows our claim.
Looking at \eqref{S_ni}, if $m>1$, then there exist two different subsets $J,J'$ of size $i$, such that, for some $\alpha \in \C$,
$ S_J - \alpha S_{J'} \in I$. However the size of $J$ is less then half of $n$, which means that the polynomial cannot have all variables $\lambda_j$ and so cannot be a scalar multiple of the generator of $I$ (and cannot be any other multiple since it has degree $1$). Thus $m=1$ and $\widetilde S_{n,i}=R_{n,i}$.




Now if $i=n/2$, observe that the polynomial $S_{n,i}$ can be written as a square $Q_{n,i}^2$ in $\C[\lambda_1,\ldots,\lambda_n]/I$. Indeed, if $J\subseteq \{1,\ldots,n\}$ satisfies $|J| = n/2$, then $|J^c| = n/2$ as well, so $S_{n,i}$ is made up of products of the form
$S_J S_{J^c} = -S_J^2$ since
$S_J + S_{J^c} = S_{\{1,\ldots,n\}}=0$.

We can write a representative for $Q_{n,i}$ by picking the subsets which contain $1$, i.e.,
\[Q_{n,i} = \prod_{\substack{J \subseteq \{2,\ldots,n\} \\ |J| = n/2-1 }} 
(
\lambda_1 +S_J
).\] 
To check that it is a symmetric polynomial, it suffices to see the action of transpositions of the form $(1 \,\, k)$, since any permutation of $\{1,\ldots,n\}$ can be decomposed as a product of such transpositions. 
When acting on each factor of $Q_{n,i}$, the variable $\lambda_1$ either remains in the permuted factor -- meaning that it gets sent to another factor of $Q_{n,i}$ -- or it disappears. But when this happens, we can replace the subset $J$ by its complement and we get a factor of $Q_{n,i}$ but with a minus sign, remembering that the sum of all $\lambda_j$ is zero. The number of such minus signs equals
\[\#\{J \subset \{2,\ldots,n\} \mid |J|=n/2-1, k \notin J\}=\binom{n-2}{n/2-1}=\binom{2(i-1)}{i-1}.\] Since this is an even number for $i \ge 2$, the polynomial is indeed symmetric. Thus, as in the previous case, there exists a polynomial $\widetilde Q_{n,i} \in \C[X_2,\ldots,X_n]$ such that
\[Q(\lambda_1,\ldots,\lambda_n)=\widetilde Q(e_2,\ldots,e_n),\] where again the $e_j$ are the elementary symmetric polynomials of the $\lambda_j$, satisfying $\widetilde S_{n,i} = \widetilde Q_{n,i}^2$. A similar argument to the case $i \ne n/2$ shows that $\widetilde Q_{n,i}$ is irreducible, thus equal to $R_{n,i}$.
\end{proof}  

\begin{lemma}\label{prime}
	The polynomial $S_J$ is prime in the quotient ring $\C[\lambda_1, \ldots, \lambda_n]/I$.
\end{lemma}

\begin{proof}
	As an element of $\C[\lambda_1, \ldots, \lambda_n]$, $S_J$ is irreducible, thus prime, since a polynomial ring over a field is a unique factorization domain. Such a property is not in general inherited when taking quotients, so while $S_J$ remains irreducible in the quotient, we need to show that it is also prime. To that effect, let $P$ and $Q$ be two polynomials such that $S_J$ is a factor of $PQ$ in the quotient ring, i.e., there exist polynomials $T,R$ such that \[PQ = TS_J + RS_{\{1,\ldots,n\}}.\] The result will be shown if the ideal generated by $S_J$ and $S_{\{1,\ldots,n\}}$ is prime in $\C[\lambda_1, \ldots, \lambda_n]$. But its zero set is the solution set of two homogeneous linear equations, hence a linear subspace of $\C^n$, which is irreducible. 
\end{proof}

We now know the (weighted) degrees of the $R_{n,i}$ so, if after evaluating them at a point of the Hitchin base, the resulting section has reduced divisor, we know how many distinct zeros it has, namely $d(n,i)(2g-2)$. Consequently, due to \eqref{ineq:divisor}, there is a finite number of choices for each divisor of $\Phi_i$, equal to \[\binom{d(n,i)(2g-2)}{\deg \Phi_i}.\]

The following result says that, for generic $a \in \mathcal A$, this is precisely what happens, along with other properties that simplify our counting problem.




\begin{prop}\label{prop:generic}
	Let $n \ge 3$ and $ 1\le i \le n-1$. For generic $a \in \mathcal A$, 
\begin{enumerate}
	\item each section $R_{n,i}(a)$ has reduced divisor;
	\item no two different sections $R_{n,i}(a)$ and $R_{n,j}(a)$ (meaning $j \notin \{i,n-i\}$) have a common zero;
	\item the zeros of $R_{n,i}(a)$ are disjoint from the ramification divisor of the spectral curve $X_a$;
	\item for $c \in \div R_{n,i}(a)$, there is a unique partition of the points of the spectral curve $X_a$ over $c$ into two blocks of sizes $i$ and $n-i$, such that the sum of points in each block is zero. 
\end{enumerate}
\end{prop}

\begin{proof}
	Since $\mathcal A$ is irreducible, it suffices to show that the space of sections satisfying each condition is open and non-empty.
	
	We begin with condition $1$. The space of sections in $\mathcal A$ satisfying it is open so it remains to show the existence of a single $a \in \mathcal A$ such that $R_{n,i}(a)$ has reduced divisor. To construct such a point we will consider the space of sections
	\[B=\left\{\lambda = (\lambda_1,\ldots,\lambda_n) \in H^0(K)^n \left| \, \sum \lambda_j = 0\right.\right\}\] and show that there exists a tuple such that the section \[S_{n,i}(\lambda) = \prod_{\substack{J \subseteq \{1,\ldots,n\} \\ |J| = i }} S_J(\lambda) \in H^0\left(K^{d(n,i)}\right)\] has reduced divisor (in the case of $n$ even and $i=n/2$, we consider the square root of $S_{n,i}$ as in Proposition \ref{R_degree}). Here $S_J:B \to H^0(K)$ is formally defined as in \eqref{S_J}. This will be enough since we then consider the tuple $(a_2,\ldots,a_n) \in \mathcal A$ given by the elementary symmetric polynomials on the element of $B$ we find.

	 We split the proof in two steps. For $\lambda \in B$ and $J$ a subset of $\{1,\ldots,n\}$ of size $i$, we will show that
	\begin{enumerate}
		\item[(A)] for $\lambda \in B$ generic, $S_J(\lambda)$ has reduced divisor for each subset $J$;
		\item[(B)] for $\lambda \in B$ generic, $S_J(\lambda)S_{J'}(\lambda)$ has reduced divisor for each pair of distinct and non-complementary subsets $J,J'$,
	\end{enumerate} which is precisely the same as $S_{n,i}(\lambda)$ having reduced divisor.

	For claim (A),  
	since the linear sistem associated to the canonical bundle $K$ of the base curve $X$ is base point free, 
	it follows by Bertini's theorem that the generic section of $K$ has reduced divisor. Denote by $U$ the open set of such sections. Now $S_J:B \to H^0(K)$ is a surjective map since, given a section $s \in H^0(K)$, there exist indices $j \in J$, $j' \notin J$ (since $|J| = i < n$) so the tuple $\lambda$ with $\lambda_{j} = s$, $\lambda_{j'}=-s$ and $0$ at all other indices belongs in $B$ and satisfies $S_J(\lambda)=s$. Thus $S_J^{-1}(U)$ is open and non-empty for each $J$ which in turn means that the finite intersection $\bigcap_{J \subset \{1,\ldots,n\}, |J|=i} S_J^{-1}(U)$ is also open and non-empty in $B$, proving the first claim.

	Moving on to claim (B), fix two distinct and non-complementary subsets $J,J' \subset \{1,\ldots,n\}$ of size $i$. 
	Consider the set 
	\[I_{J,J'} \coloneqq \{(x,\lambda) \mid S_J(\lambda)(x) = S_{J'}(\lambda)(x)=0\} \subset X \times B.\] Let $x \in X$ and consider the map
	\[\ev_{x,J,J'}: B \longrightarrow K_x \oplus K_x \] given by $\lambda \mapsto (S_J(\lambda)(x),S_{J'}(\lambda)(x))$. We now show that it is surjective. If $i\ne n/2$ we can assume $i < n/2$---since $R_{n,i}=R_{n,n-i}$---so that there exist indices $j\in J\setminus J'$, $j' \in J'\setminus J$ and $k \notin J \cup J'$. 
	Given a pair $(\alpha,\beta)$ of sections in $H^0(K)$ consider the tuple $\lambda$ defined by $\lambda_j = \alpha$, $\lambda_{j'}=\beta$, $\lambda_{k}=-\alpha-\beta$ and $0$ at all other indices. If $i=n/2$, since $J' \ne J^c$ we can again find some $k \notin J \cup J'$ so the same argument works. 

	Thus, $\ev_{x,J,J'}$ is surjective, which implies that
	\[\codim_B(\ker \ev_{x,J,J'})=2\] and, since, by definition, $\ker \ev_{x,J,J'} = I_{J,J'} \cap (\{x\}\times A)$ it follows that 
	\[\codim_{X \times B}(I_{J,J'})=2.\] Considering the projection $p_B: X \times B \to B$, it follows that \[\codim_B (p_B(I_{J,J'})) \ge 1.\] Its complement then contains an open dense set $V_{J,J'}$, whose elements are tuples $\lambda \in B$ such that $S_J$ and $S_{J'}$ have no zeros in common. Finally, the finite intersection $V = \bigcap_{J,J'}V_{J,J'}$ is also open and dense, so the set of $\lambda \in B$ such that no two of the $S_J$ have a common zero is generic.

	Hence it is indeed possible to find a $\lambda \in B$ satisfying both conditions (A) and (B), as we wanted.
	
	To show that condition number $2.$ is generic, the goal now is to show that the space of sections $\lambda \in B$ such that $S_{n,i}(\lambda)$ and $S_{n,j}(\lambda)$ have no common zeros is generic. Such a zero would have to be a common zero of $S_{J}(\lambda)$ and $S_{J'}(\lambda)$, with $|J|=i$ and $|J'|=j$. We can then use the same argument for claim (B) of item $1.$, now allowing the sizes of $J$ and $J'$ to be different---we can assume $i<j\le n/2$.
	
	As for item $3.$, let $c \in X$ be such that $\lambda_k(c)=\lambda_l(c)$, for some distinct $k$ and $l$. Then there exist two distinct subsets $J$, $J' \subseteq \{1,\ldots,n\}$ of size $i$, satisfying $S_J(\lambda)(c)=S_{J'}(\lambda)(c)=0$, which implies that $c$ is a zero of order at least two of $S_{n,i}(\lambda)=R_{n,i}(a)$. But we have already seen that for generic $a$ such a zero cannot exist.
	

	Finally, for item $4.$, if $i<n/2$, notice that if $\lambda \in B$ is such that $S_{n,i}(\lambda)$ has only simple zeros then at each zero $c$ there can be only one subset $J$ of size $i$ such that $S_J(\lambda)(c)=0$, which shows the result (if $i=n/2$, there is a unique pair of subsets $J,J'$ such that $S_J(c)=S_{J'}(c)=0$). \qedhere

\end{proof}


We then have the following bounds for the number of intersection points by multiplying the choices for each divisor of $\Phi_i$, the number of $n$-th roots of a line bundle and the number of eigenvalue choices over each zero of every $\Phi_i$, which is $1$ except at $i = n/2$, when it becomes $2$.

\begin{cor}\label{cor:bounds}
	Let $F_\alpha$ be a fixed point component of type $(1,\ldots,1)$ with $\deg \Phi_i = m_i$. Then, for $a \in \mathcal A$ in the conditions of Proposition \ref{prop:generic}:
	\begin{itemize}
		\item the number of points in $F_\alpha(a)$ (defined in \ref{admissible}) is bounded above by
		\[n^{2g}\prod_{i=1}^{\frac{n-1}{2}}\binom{\binom{n}{i}(2g-2)}{m_i,m_{n-i}}\]
			 for $n$ odd, and by  \[n^{2g}\binom{\binom{n}{n/2}(g-1)}{m_{n/2}}\prod_{i=1}^{\frac{n}{2}-1}\binom{\binom{n}{i}(2g-2)}{m_i,m_{n-i}}\]
		for $n$ even. Furthermore, these are all very stable fixed points;
		
		\item for $\mathcal E \in F_\alpha(a)$, we have
		\[|\mathcal L_{\mathcal E} \cap h^{-1}(a)| \le 1, \text{ if $n$ odd}\] and
		\[|\mathcal L_{\mathcal E} \cap h^{-1}(a)| \le 2^{m_{n/2}}, \text{ if $n$ even.}\]
	\end{itemize}

\end{cor}

\begin{proof}
	 Given such $a \in \mathcal A$, by \eqref{ineq:divisor}, the number of possible $\Phi_i$ is obtained by counting the number of ways to distribute the zeros of each section $R_{n,i}(a)$ into $\Phi_i$ and $\Phi_{n-i}$ (since $R_{n,i}=R_{n,n-i}$), if $i<n/2$. If $i = n/2$, we distribute the zeros of $R_{n,n/2}$ over $\Phi_{n/2}$ only. The conditions of Proposition \ref{prop:generic} guarantee that the number of ways to obtain such a distribution is, for $n$ odd,
	\[\prod_{i=1}^{\frac{n-1}{2}}\binom{\binom{n}{i}(2g-2)}{m_i,m_{n-i}}\]
	and, for $n$ even \[\binom{\binom{n}{n/2}(g-1)}{m_{n/2}}\prod_{i=1}^{\frac{n}{2}-1}\binom{\binom{n}{i}(2g-2)}{m_i,m_{n-i}}.\] Each such choice of divisors gives rise to $n^{2g}$ fixed points.
	
	Moreover, Proposition \ref{prop:generic} also guarantees that $\sum \div \Phi_i$ is reduced, which means that every fixed point in $F_\alpha(a)$ is very stable, by Theorem \ref{HH_vs}.
	
	Finally, by the third part of Proposition \ref{prop:generic}, at each zero of each $\Phi_i$ there is a unique partition of its eigenvalues into 2 blocks of sizes $n-i$ and $i$ and so, according to \ref{HH_count}, we must pick the one of size $(n-i)$, which means we only get one choice if $i \ne n/2$ and $2$ choices if $i=n/2$. This gives us the stated bounds on the number of points in $\mathcal L_{\mathcal E} \cap h^{-1}(a)$.
\end{proof}



\begin{theorem}
\label{thm:B}
\ThmB
\end{theorem}

\begin{proof}
It remains to see that, given each choice of fixed point $\mathcal E$ with $\deg \Phi_i = m_i$ and eigenvalues over the zeros of $\Phi_i$ discussed in Corollary \ref{cor:bounds}, we can produce a point in $\mathcal L_{\mathcal E} \cap h^{-1}(a)$.

	As in the proof of Proposition \ref{HH_count}, we proceed by induction on the total degree $\deg \delta_+$. If it is zero, then we are in the Hitchin section, for which $\mathcal L_{\mathcal E} = W_{\mathcal E}^+$ and so the existence of the intersection point is clear. 
	
	Otherwise, assume we have produced a point $(E,\Phi) \in \mathcal L_{\mathcal E} \cap h^{-1}(a)$, where $\mathcal E$ satisfies $\deg \delta_+=(\sum_i m_i) - 1$ and we need to add a zero at $\Phi_i$, with $1\le i \le n-1$. We pick a point $c \in \div R_{n,i}(a)$, where the eigenvalues of $\Phi_c$ partition into two blocks each summing to zero, one of size $i$ and another of size $n-i$, according to Proposition \ref{prop:generic}. Letting $V < E_c$ be the span of the $n-i$ eigenvectors corresponding to the second block of eigenvalues, we have $\tr(\Phi_c|_V)=0$ so, by Proposition \ref{zv_hecke}, the Hecke transformed Higgs bundle $\mathcal H_V(E,\Phi)$ belongs to the zero velocity locus $\mathcal L_{\mathcal E'}$, where $\mathcal E'$ is a fixed point whose associated divisor $\delta_+'$ satisfies $\deg \delta'_+ = \sum_i m_i$.

Finally, we check with which Hitchin section we must start the algorithm. We seek a line bundle $L$ such that performing the Hecke transformations to the point in $h^{-1}(a)$ flowing to
\begin{equation*}
	\begin{tikzcd}
		L \arrow[r,"1"] & LK^{-1} \arrow[r,"1"] & \cdots \arrow[r,"1"] & LK^{1-n} 
	\end{tikzcd}
\end{equation*}
gives us a point flowing to
\begin{equation*}
	\begin{tikzcd}
		L_1 \arrow[r,"\Phi_1"] & L_2 \arrow[r,"\Phi_2"] & \cdots \arrow[r,"\Phi_{n-1}"] & L_{n} 
	\end{tikzcd}
\end{equation*}
with trivial determinant, i.e. $L_1\cdots L_n \cong \mathcal O$. After performing the Hecke transformations, we flow to the point
\[L(-\div \Phi_1 - \cdots -\div \Phi_{n-1}) \oplus \cdots \oplus LK^{-j}(-\div \Phi_{j+1} - \cdots -\div \Phi_{n-1}) \oplus \cdots \oplus LK^{1-n}\] whose determinant equals $L^n K^{-(n-1)n/2}(L_n^{-n})$.
So we can take $L=K^{(n-1)/2}L_n$.
\end{proof}

\begin{rmk}
	After fixing the divisors of the $\Phi_i$, the fact that we are in the $\SL(n,\C)$ moduli space limits the choices of fixed points to a finite number. If instead we allow ourselves to move inside the full $\GL(n,\C)$ moduli space, we get a full copy of $\Jac X$ for each choice of divisors, which then lifts to $|\mathcal L_\alpha \cap h^{-1}(a)$ copies inside $\mathcal L_\alpha$ when intersecting with $h^{-1}(a)$. 
\end{rmk}

\section{A proposal for the mirror} \label{mirror}

The Lagrangian $\mathcal L_\alpha$ we have constructed together with its structure sheaf is a brane of type BAA with respect to the hyperkähler structure of the moduli space $\mathcal M$. Thus according to \cite{KW}, it should correspond to a BBB brane on the mirror moduli space $\mathcal M^\vee$ -- which will be the moduli of $\PGL(n,\C)$-Higgs bundles---i.e., a hyperholomorphic bundle over a hyperkähler subvariety. Moreover, the correspondence should be realized as a Fourier-Mukai transform along each fiber of the Hitchin map.


In this section, after recalling the machinery of Fourier-Mukai transforms, we review the construction of the mirror of the upward flow of a very stable type $(1,\ldots 1)$ fixed point done in \cite[Section 6.2]{HH}. We then see what modifications will be needed to obtain a similar result for our Lagrangian $\mathcal{L}_\alpha$.
\subsection{The Fourier-Mukai transform}

In this section, we recall the main concepts and results needed to define the Fourier-Mukai transform of an abelian variety. 


Let $A$ be an abelian variety and $A^\vee = \Pic^0(A)$ its dual. 
The Poincaré bundle $\mathcal P$ over the product $A \times A^\vee$ is the unique line bundle satisfying 
\[\mathcal P|_{A \times \{L\}} \cong L,\qquad \mathcal P|_{\{0 \} \times J^\vee} \cong \mathcal O_{A^\vee}.\] Moreover, given $x,y \in A$, if we set $\mathcal P_x \coloneqq \mathcal P|_{\{x\} \times A^\vee}$, the biextension property of $\mathcal P$ (see, e.g., \cite[Section 10.3]{Polishchuk_2003}) implies that 
\begin{equation}\label{biext}
	\mathcal P_{x+y} \cong \mathcal P_x \otimes \mathcal P_y.
\end{equation}


Letting $D^b_{coh}(A)$ denote the derived category of bounded complexes of coherent sheaves on $A$, the Fourier-Mukai transform  $S:D^b_{coh}(A) \to D^b_{coh}(A^\vee)$ is defined by \cite{mukai81}
\[S(\mathcal F^\bullet) \coloneqq \pi_{2,*}(\pi_1^*\mathcal F^\bullet \otimes^L \mathcal P),\] where $\pi_1,\pi_2$ are the projections of $A \times A^\vee$ onto its factors.

We now specialize to the case of the Jacobian of a Riemann surface $X$, which we denote by $J\coloneqq \Jac(X)$, parametrizing the degree 0 line bundles over $X$. It is an abelian variety when equipped with the tensor product of line bundles. 
Given a point $c_0 \in X$, we have the \emph{Abel-Jacobi map} $\alpha_{c_0}:X \to J$ defined by $c \mapsto \mathcal O(c-c_0).$ 
Its pullback $\alpha_{c_0}^* : J^\vee \to J$ is an isomorphism, independent of the base point $c_0$, which we will use to identify $J$ and $J^\vee$. 


Given $L \in J$, we have
\begin{equation}\label{LM_point}
	S(\mathcal O_L) = \mathcal P_L.
\end{equation}
Moreover, given $L_1,L_2 \in J$, equation \eqref{biext} becomes
\begin{equation}\label{P_prod}
\mathcal P_{L_1 \otimes L_2}\cong \mathcal P_{L_1}\otimes \mathcal P_{L_2}.
\end{equation}
Letting $D$ be a divisor on $X$ such that $L=\mathcal O (D)$, we write $\mathcal P_D$ for $\mathcal P_{L}$ and \eqref{P_prod} becomes
\begin{equation}\label{FM_sum}
	\mathcal P_{D_1 + D_2} = \mathcal P_{D_1} \otimes \mathcal P_{D_2}.
\end{equation}

We can pullback the Poincaré bundle $\mathcal P$ over $J\times J^\vee$ using the map
$\alpha_{c_0} \times (\alpha_{c_0}^*)^{-1}$ to obtain the universal line bundle $\mathbb L$ over $X \times J$ normalized at $c_0$, i.e., for all $L \in J$,
\begin{equation}\label{universal_line}
	\mathbb L|_{X \times \{L\}} \cong L, \qquad \mathbb L|_{\{c_0\}\times J} \cong \mathcal O_J.
\end{equation}


\subsection{Construction via universal Higgs bundles}


In \cite[Section 6.2.1]{HH} a twisted universal Higgs bundle $(\mathbb E,\mathbf{\Phi})$ is built over $\mathcal M^s \times X$ (recall that $\mathcal M^s$ is the smooth locus of $\mathcal M$, represented by the stable points). Over an étale cover $(U_i \to \mathcal M^s)$ we have universal Higgs bundles $\mathbb E_{U_i}$ together with isomorphisms $\phi_{ij}: \mathbb E_{U_i} \to \mathbb E_{U_j}$ over $U_i \times_{\mathcal M^s}U_j$. These satisfy, for each $i,j,k$, \[\phi_{jk}\phi_{ij} = \theta_{ijk}\phi_{ik},\] where $\theta_{ijk}$ is an $n$-th root of unity. The collection $\theta$ of the $\theta_{ijk}$ is in fact a Čech 2-cocycle defining a cohomology class $[\theta] \in  H^2(\mathcal M^s,\mathcal O^*_{\mathcal M^s})$, i.e., a gerbe. Thus $\mathbb E$ is only a vector bundle in the twisted sense.

 Recall that, by \eqref{divisors}, a very stable $\C^*$-fixed point $\mathcal E$ of type $(1,\ldots,1)$ corresponds to a vector of divisors on $X$,
\[\delta = (\delta_0;\delta_1,\ldots,\delta_{n-1}) \in \Jac^l(X) \times X^{[m_1]} \times \cdots \times X^{[m_{n-1}]},\] 
where $\sum_{i>0} \delta_i$ is reduced. We write $\delta_i = \sum_{j=1}^{m_i} c_{ij}$, where $c_{ij} \in X$ if $i>0$ but if $i=0$, either $c_{0j} \in X$ or $-c_{0j} \in X$. We also set $L=\mathcal O(\delta_0)$. Associated to $\mathcal E$ (or equivalently, to $\delta$), we consider
\begin{equation}\label{lambda_E}
	\Lambda_{\mathcal E} \coloneqq \bigotimes_{i=0}^{n-1} \bigotimes_{j=1}^{m_i} \Lambda^{n-1}(\mathbb E_{c_{ij}}), 
\end{equation} where $\mathbb E_{c_{ij}} \coloneqq \mathbb E|_{\mathcal M^s \times \{c_{ij}\}}$ if $c_{ij}\in X$ and $\mathbb E_{c_{0j}} \coloneqq \mathbb E^*|_{\mathcal M^s \times \{-c_{0j}\}}$
if $-c_{0j} \in X$.

Since $\mathbb E$ is twisted by a gerbe $\theta$, $\Lambda_{\mathcal{E}}$ is naturally a $\theta^D$-twisted bundle, with $D=nl +\sum_i (n-i)m_i$. We have
\begin{equation}\label{d=D}
	\deg \mathcal E = D - n(n-1)(g-1),
\end{equation}
so $D$ is a multiple of $n$, since $\mathcal E$ has degree $0$, being a point of $\mathcal M$. Since $\theta$ takes values in the $n$-th roots of unity, we conclude that $\Lambda_{\mathcal{E}}$ is a well defined vector bundle on $\mathcal M^s$.


Finally, in \cite[Section 6.2.2]{HH}, the relative Fourier-Mukai of $\mathcal O_{W_{\mathcal E}^+}$ is computed over the open subset $\mathcal A^\#$ of the base $\mathcal A$ of points satisfying the conditions of Proposition \ref{HH_count}.

\begin{theorem}[{\cite[Theorem 6.3]{HH}}] \label{HH_mirror}
	The relative Fourier-Mukai transform over $\mathcal A^\#$ satisfies
	\[S\left(\left.\mathcal O_{W_{\mathcal E}^+}\right|_{\mathcal M^\#}\right) = \Lambda_{\mathcal E}|_{\mathcal M^\#}.\] Here $\mathcal M^\# = h^{-1}(\mathcal A^\#)$. 
\end{theorem}

Denote by $\mathcal A^0$ the subset of $\mathcal A^\#$ whose points satisfy the conditions of Proposition \ref{prop:generic}. For $a \in \mathcal A^0$, the structure sheaf of the intersection $\mathcal L_\alpha \cap h^{-1}(a)$ will be a skyscraper sheaf supported on a finite number of points in $\Prym_a$, the Prym variety associated to the spectral cover $\pi:X_a \to X$. It is defined as the kernel of the norm map $\Nm_\pi: \Jac(X_a) \to \Jac(X)$ which sends $\mathcal O(\sum_i n_i \tilde c_i)$ to $\mathcal O(\sum_i n_i \pi(\tilde c_i))$.

We can consider the structure sheaf of the intersection $\mathcal L_\alpha \cap h^{-1}(a)$ as a skyscraper sheaf on the entire Jacobian $J_a \coloneqq \Jac(X_a)$, supported on the same points, and compute its dual using the Fourier-Mukai transform associated to $J_a$. In this way, we obtain a vector bundle over each $J_a^\vee \cong J_a$ which we then show that it can be extended to a vector bundle $\Lambda_\alpha^0$ over the subset of $\mathcal M (\GL(n,\C))$ defined as $\mathcal M^0_{\GL} \coloneqq h_{\GL}^{-1}(\{0\} \times \mathcal A^0)$. This will in turn descend to a vector bundle $\overline \Lambda_\alpha^0$ on $\mathcal M^0_{\PGL} \coloneqq h_{\PGL}^{-1}(\mathcal A^0) \subseteq \mathcal M(\PGL(n,\C))$. The strategy can be summarized in the following diagram
\[
\begin{tikzcd}
	\mathcal L_\alpha \cap \mathcal M^0 \arrow[d,hook] \arrow[r,equal] & \mathcal L_\alpha \cap \mathcal M^0 \arrow[d,hook] & & \Lambda^0_\alpha \arrow[d] \arrow[r] & \overline \Lambda_\alpha^0 \arrow[d] \\ 
	\mathcal M^0 \arrow[rrdd,"h",swap] \arrow[r,hook]& \mathcal M^0_{\GL} \arrow[rd,"h_{\GL}",swap] & & \mathcal M^0_{\GL} \arrow[ld,"h_{\GL}"] \arrow[r,two heads] & \mathcal M^0_{\PGL}  \arrow[lldd,"h_{\PGL}"] \\ 
	& & \{0\} \times \mathcal A^0 \arrow[d, equal, "\sim", sloped] & & \\ 
		& &  \mathcal A^0. & &
\end{tikzcd}
\]

Let $F_\alpha$ be a $\C^*$-fixed point component of type $(1,\ldots,1)$, let $a \in \mathcal A^0$ and let $F_\alpha(a)$ be the corresponding set of admissible fixed points. Each $\mathcal E \in F_\alpha(a)$ has an associated divisor vector $(\delta_0;\delta_1,\ldots,\delta_n)$, with $\delta_i = \sum_i c_{ij}$, such that $\delta_+$ is reduced. For $i > 0$, over each $c_{ij}$ there are $n$ distinct preimages of the spectral cover $\pi:X_a \to X$. For $i\ne n/2$, let $\Sigma_{ij}$ be the unique subset of size $(n-i)$ of those preimages whose sum equals $0$ and if $i=n/2$, denote the two size $(n-i)$ subsets with the same property by $\Sigma_{\frac{n}{2}j}^0$ and $\Sigma_{\frac{n}{2}j}^1$. 

The structure sheaf of the intersection $\mathcal L_\alpha \cap h^{-1}(a)$ can be written as
\[\mathcal O_{\mathcal L_\alpha \cap h^{-1}(a)} = \bigoplus_{\mathcal E \in F_\alpha(a)} \mathcal O_{\mathcal L_{\mathcal E} \cap h^{-1}(a)}.\]   By Corollary \ref{cor:bounds} and Theorem \ref{thm:B}, if $n$ is odd, the intersection $\mathcal L_{\mathcal E} \cap h^{-1}(a)$ consists of a single point, namely,
\begin{equation}\label{n_odd}
	\pi^*L\left(\sum_{i=1}^{n-1} \sum_{j=1}^{m_i} \sum_{\tilde c \in \Sigma_{ij}}\tilde c\right) \in h^{-1}(a) \cong J_a^d,
\end{equation}
where $d= n(n-1)(g-1)$. If $n$ is even, we have to distinguish the case $i=n/2$, which leads to $2^{m_{n/2}}$ points in $\mathcal L_{\mathcal E} \cap h^{-1}(a)$ of the form
\[\pi^*L\left(\sum_{\substack{i=1 \\ i\ne n/2}}^{n-1} \sum_{j=1}^{m_i} \sum_{\tilde c \in \Sigma_{ij}}\tilde c + \sum_{j=1}^{m_{n/2}}\sum_{\tilde c \in \Sigma_{\frac{n}{2}j}^{k_j}} \tilde c\right),\] one for each possible choice of $k_j \in \{0,1\}$. 


To perform the Fourier-Mukai transform of $\mathcal O_{\mathcal L_\alpha \cap h^{-1}(a)}$, we first choose a point $c_0 \in X$ and a lift $\widetilde{c_0}$ to the spectral curve $X_a$. We then identify $J_a^d$ with $J_a$ using the map $M \mapsto M(-d\tilde c_0)$. For $n$ odd, \eqref{n_odd} becomes, using \eqref{d=D}, 
\begin{equation*}
	\pi^*L(-nl\tilde c_0)\left(\sum_{i=1}^{n-1} \sum_{j=1}^{m_i} \sum_{\tilde c \in \Sigma_{ij}}(\tilde c-\tilde c_0)\right) \in  J_a.
\end{equation*}

Using \eqref{LM_point} and \eqref{FM_sum}, we obtain the Fourier-Mukai transform of $\mathcal O_{\mathcal L_\alpha \cap h^{-1}(a)}$,
\[S(\mathcal O_{\mathcal L_\alpha \cap h^{-1}(a)})=\bigoplus_{\mathcal E \in F_\alpha(a)} \widetilde{\mathcal P}_{\pi^* L(-nl\tilde c_0)} \bigotimes_{i=1}^{n-1} \bigotimes_{j=1}^{m_i}\Lambda^{n-i}
\Lp\bigoplus_{\tilde c \in \Sigma_{ij}}\widetilde{\mathcal P}_{\tilde c-\tilde c_0}\Rp.\]  

For $n$ even, we have to distinguish the $i=n/2$ case, so the Fourier-Mukai transform $S(\mathcal O_{\mathcal L_\alpha \cap h^{-1}(a)})$ takes the form
\[\bigoplus_{\mathcal E \in F_\alpha(a)} \widetilde{\mathcal P}_{\pi^* L(-nl\tilde c_0)} \bigotimes_{\substack{i=1 \\ i\ne \frac{n}{2}}}^{n-1} \bigotimes_{j=1}^{m_i}\Lambda^{n-i}\Lp\bigoplus_{\tilde c \in \Sigma_{ij}}\widetilde{\mathcal P}_{\tilde c-\tilde c_0}\Rp \bigotimes_{j=1}^{m_{\frac{n}{2}}} \Lp \Lambda^{\frac{n}{2}}\Lp\bigoplus_{\tilde c \in \Sigma^0_{\frac{n}{2}j}}\widetilde{\mathcal P}_{\tilde c-\tilde c_0}\Rp \oplus \Lambda^{\frac{n}{2}}\Lp\bigoplus_{\tilde c \in \Sigma_{\frac{n}{2}j}^1}\widetilde{\mathcal P}_{\tilde c-\tilde c_0}\Rp\Rp. \]

Our goal now is to find a vector bundle over $\mathcal M^0$ whose restriction to each fiber $h^{-1}(a)$ is given by $S(\mathcal O_{\mathcal L_\alpha \cap h^{-1}(a)})$. The way we will construct such a bundle will be by relating the expression for the Fourier-Mukai transform to data that can be determined canonically from the universal Higgs bundle $(\mathbb E,\mathbf \Phi)$.

From the proof of Theorem \ref{HH_mirror} in \cite{HH}, we have the following facts:
\begin{itemize}
	\item for $i\ge 1$, 
	\[\widetilde{\mathcal P}_{\tilde c-\tilde c_0} = \widetilde {\mathbb L}|_{J_a \times \{\tilde c\}},\] where $\widetilde {\mathbb L}$ is the universal line bundle on $J_a \times X_a$ normalized at $\tilde c_0$ as defined in \eqref{universal_line}. In this way we get a universal Higgs bundle on $h^{-1}(a) \times X$ given by $(\pi_* \widetilde{\mathbb L},\mathbf{\Phi}_{a})$, with $\mathbf{\Phi}_a \coloneqq \pi_*(x: \widetilde{\mathbb L} \to \widetilde{\mathbb L}\otimes \pi^*K)$, which satisfies $\pi_* \widetilde{\mathbb L} \cong \mathbb E$ over $h^{-1}(a)$;
	\item using the following normalized determinant map
	\[
	\begin{array}{crcl}
		\operatorname{Ndet}: & \mathcal M^0 & \longrightarrow & J \\
		& (E,\Phi) & \longmapsto & \det(E)\otimes K^{\binom{n}{2}},
	\end{array}
	\] over $h^{-1}(a)$ the following bundles are isomorphic
	\begin{equation} \label{E_L}
		\widetilde{\mathcal P}_{\pi^* L(-nl\tilde c_0)} \cong \Ndet^*\mathcal P_{L(-lc_0)} \cong \bigotimes_{j=1}^{m_0}\Lambda^n\mathbb E_{c_{0j}}.
	\end{equation} In particular, this last bundle only depends on $L$ and not on the choice of divisor $\delta_0$ such that $L=\mathcal O(\delta_0)$. Set $\mathbb E_L \coloneqq \bigotimes_{j=1}^{m_0}\Lambda^n\mathbb E_{c_{0j}}$.
\end{itemize}

Note that for each $i,j > 0$ we get an inclusion
\[\Lambda^{n-i}\left(\bigoplus_{\tilde c \in \Sigma_{ij}} \widetilde{\mathcal P}_{\tilde c-\tilde c_0} \right) \hookrightarrow \Lambda^{n-i}\left(\bigoplus_{\tilde c \in \pi^{-1}(c_{ij})} \widetilde{\mathcal P}_{\tilde c-\tilde c_0} \right) \cong \Lambda^{n-i}(\mathbb E_{c_{ij}}),\] so, over $h^{-1}(a)$, our proposed mirror is a sum, for each $\mathcal E \in F_\alpha(a)$, of a vector subbundle $\Lambda^0_{\mathcal E}(a)$ of $\Lambda_{\mathcal E}|_{h^{-1}(a)}$ in \eqref{lambda_E}, given by 

\begin{equation} \label{lambda_0}
	\Lambda^0_{\mathcal E}(a) \coloneqq \mathbb E_L \otimes \bigotimes_{i=1}^{n-1} \bigotimes_{j=1}^{m_i}\bigoplus_{\Sigma_{ij}}\Lambda^{n-i}\left(\bigoplus_{\tilde c \in \Sigma_{ij}}\widetilde{\mathcal P}_{\tilde c-\tilde c_0}\right),
\end{equation} where the first direct sum is over all subsets $\Sigma_{ij}$ of size $k$ of $\pi^{-1}(c_{ij})$ whose sum is zero.

The idea now is to be able to capture, at a point $c \in X$, the subspaces of $E_c$ where $\Phi_c$ restricts to a traceless endomorphism. Given an endomorphism $A:V \to V$ of a finite dimensional vector space $V$ and $k \le \dim V$, we consider the map on the exterior product $d\Lambda^k A: \Lambda^k V \to \Lambda^k V$ defined by 
\begin{equation}\label{ext}
	d\Lambda^k A (v_1 \wedge \ldots \wedge v_k)=\sum_{i=1}^k v_1 \wedge \ldots \wedge Av_i \wedge \ldots \wedge v_k.
\end{equation}

\begin{prop}\label{ext_product}
	Let $(\mathbb E,\mathbf \Phi)$ be the universal Higgs bundle over $X \times \mathcal M^0$. Let $a \in \mathcal A^0$ and $\mathcal E \in F_\alpha(a)$ with associated divisors $\delta_i = \sum_j c_{ij}$. Then, over $h^{-1}(a)$, for $k=1,\ldots,n-1$,
	\[ \ker d\Lambda^{k} \Phi_{c_{ij}} \cong \bigoplus_{\Sigma_{ij}}\Lambda^k \left(\bigoplus_{\tilde c \in \Sigma_{ij}} \widetilde{\mathcal P}_{\tilde c -\tilde c_0}  \right),\] where the direct sum is over all subsets $\Sigma_{ij}$ of size $k$ of the eigenvalues of $\Phi_{c_{ij}}$ whose sum is zero.
\end{prop}

\begin{proof}
Given a Higgs bundle $(E,\Phi) \in h^{-1}(a)$ and a point $c_{ij}$, we are led to the following linear algebra problem. Let $V$ be a rank $n$ vector space and $A: V \to V$ be a diagonalizable endomorphism with distinct eigenvalues. Consider the associated endomorphism on the exterior product $d\Lambda^k A: \Lambda^k V \to \Lambda^k V$ defined by \eqref{ext}. 
We wish to determine the kernel of this map. 

Let $(e_i)$ be a basis of $V$ made up of eigenvectors for $A$. 
Then there is an associated basis for $\Lambda^k V$ \[(e_{i_1}\wedge \ldots \wedge e_{i_k})_{1\le i_1 < \cdots < i_k \le n}.\] If $\lambda_j$ is the eigenvalue associated to $e_j$, then, using multi-index notation, we have
\[\Lambda^kA(e_I) = \sum_{j \in I} \lambda_j \cdot e_I.\] Thus, a basis for the kernel is given by the $e_I$ where $I$ is such that $\sum_{j \in I} \lambda_j = 0$.
\end{proof}

%
%


This allows us to rewrite the bundle $\Lambda^0_{\mathcal E}(a)$ of \eqref{lambda_0} as 


\begin{equation} \label{lambda_zero}
	\mathbb E_L \otimes \bigotimes_{i=1}^{n-1} \bigotimes_{j=1}^{m_i} \ker (d\Lambda^{n-i} \Phi_{c_{ij}}).
\end{equation}

The goal now is to construct a bundle $\Lambda_\alpha^0$ over $\mathcal M^0$ whose restriction to a fiber $h^{-1}(a)$ of the Hitchin map over $\mathcal A^0$ agrees with
\begin{equation}\label{lambda_zero_ha}
	S(\mathcal O_{\mathcal L_\alpha \cap h^{-1}(a)})=\bigoplus_{\mathcal E \in F_\alpha(a)} \Lambda^0_{\mathcal E}(a).
\end{equation}

The first step towards such a construction will be to extend the definition of the bundle $\Lambda_{\mathcal E}$ of Theorem \ref{HH_mirror} to a bundle over $\mathcal M^0 \times F_\alpha$. Essentially we want to show that the dependency of $\Lambda_{\mathcal{E}}$ on the fixed point $\mathcal E$ is holomorphic. 

\begin{prop}
	There exists a vector bundle $\Lambda_\alpha$ over $\mathcal M^0 \times F_\alpha$ such that, for each $\mathcal E \in F_\alpha$, we have
	\[(\Lambda_\alpha)|_{\mathcal M^0 \times \{\mathcal E\}} \cong \Lambda_{\mathcal E}.\]
\end{prop}

\begin{proof}
	Let $i>0$. For $j=1,\ldots,m_i$, let $X_j \cong X$ and consider the product $\mathcal M^0 \times X_1 \times \cdots \times X_{m_i}$, with projection $p_j$ to $\mathcal M^0 \times X_j$. Over each $\mathcal M^0 \times X_j$ we have the (twisted) bundle $\Lambda^i \mathbb E$. We can then form the bundle 
	\[\Lambda^{(i)} \coloneqq \bigotimes_{j=1}^{m_i} p_j^*\Lambda^{n-i} \mathbb E\] over $\mathcal M^0 \times X_1 \times \cdots \times X_{m_i}$. At a tuple $c=(c_1,\ldots,c_{m_i}) \in X^{m_i}$, it satisfies
	\[\Lambda^{(i)}|_{\mathcal M^0 \times \{c\}} = \bigotimes_{j=1}^{m_i} \Lambda^{n-i}(\mathbb E_{c_j}). \] Given any permutation $\sigma$ of the $c_j$, we have $\Lambda^{(i)}|_{\mathcal M^0 \times \{c\}} \cong \Lambda^{(i)}|_{\mathcal M^0 \times \{\sigma \cdot c\}}$ so that $\Lambda^{(i)}$ descends to a bundle over the symmetric product  $\mathcal M^0 \times X^{[m_i]}$.
	
	 We have a map from $\mathcal M^0 \times J^l$ to $J\times J$ defined by
	\[((E,\Phi),L) \longmapsto (\Ndet(E,\Phi),L(-lc_0))\] which we can use to pull back the Poincaré bundle $\mathcal P$ over $J\times J$ and obtain a bundle $\Lambda^{(0)}$ over $\mathcal M^0 \times J^l$. Using \eqref{E_L}, it has the following property
	\[\Lambda^{(0)}|_{\mathcal M^0 \times \{L\}} \cong \Ndet^*(\mathcal P_{L(-lc_0)}) \cong \bigotimes_{j=0}^{m_0} \Lambda^n(\mathbb E_{c_{0j}}),\] where $L = \mathcal O(\delta_0)$, with $\delta_0 = \sum_{j=1}^{m_0} c_{0j}$ and $\pm c_{0j} \in X$.
	
	We can then build a bundle $\Lambda_\alpha$ over the product $\mathcal M^0 \times J^l(X) \times X^{[m_1]} \times \cdots \times X^{[m_{n-1}]}$ by pulling back and tensoring all of the $\Lambda^{(i)}$, with $i=0,\ldots,n-1$, 
	which, after restricting to $\mathcal M^0 \times F_\alpha$, satisfies the property in the statement of the theorem.
\end{proof}

\begin{prop}
	There exists a vector bundle $\Lambda^0_\alpha$ over $\mathcal M^0$ whose restriction to each fiber $h^{-1}(a)$ equals 
	\[
	S(\mathcal O_{\mathcal L_\alpha \cap h^{-1}(a)}) = \bigoplus_{\mathcal E \in F_\alpha(a)} \Lambda^0_{\mathcal E}(a).
	\]
	
\end{prop}

\begin{proof}

We now consider the following subvariety of $\mathcal M^0 \times F_\alpha$, 
\[\widetilde{\mathcal M} \coloneqq \{((E,\Phi),\mathcal E) \mid \mathcal E \in F_\alpha(h(E,\Phi))\},\] and set $\widetilde \Lambda \coloneqq \Lambda|_{\widetilde{\mathcal M}}$. The projection onto the first factor is a finite covering $p:\widetilde{\mathcal M} \to \mathcal M^0$ whose fiber over a point $(E,\Phi) \in h^{-1}(a)$ is precisely the set of admissible fixed points $F_\alpha(a)$.

Recall that if $\mathcal E \in F_\alpha(h(E,\Phi))$ then, for each $c_{ij} \in \div \Phi_i$, the eigenvalues of the Higgs field $\Phi_{c_{ij}}$ split into two blocks of size $i$ and $n-i$, each summing to zero. In turn, this implies the existence of a $(n-i)$-dimensional $\Phi_{c_{ij}}$-invariant subspace where the restriction has trace zero. In terms of the universal bundle $(\mathbb E,\mathbf \Phi)$, this can be interpreted as saying that the induced map $\Lambda^{n-i} \mathbf \Phi_{c_{ij}} : \Lambda^{n-i} \mathbb E_{c_{ij}} \to \Lambda^{n-i} \mathbb E_{c_{ij}}$ has a non-zero kernel, by Proposition \ref{ext_product}.


For each $i$, the map $\Lambda^{n-i} \mathbf \Phi$ is an endomorphism of $\Lambda^{n-i}\mathbb E$ over $\mathcal M^0 \times X$. Consider the map $R_{n,i}: \mathcal A \times X \to \C$ defined in a natural way by $(a,c) \mapsto R_{n,i}(a)(c)$. We then have a universal divisor $Z_{n,i} \coloneqq \div R_{n,i}$ which can be pulled back using the map $h = h \times \id: \mathcal M \times X \to \mathcal A \times X$ to obtain 
\[h^* Z_{n,i} = \{((E,\Phi),c) \mid c \in \div R_{n,i}(h(E,\Phi))\}.\] Over this divisor of $\mathcal M \times X$, the map $\Lambda^{n-i} \mathbf \Phi$ has constant rank, so its kernel $\mathbb{V}_i$ is a subbundle over $h^* Z_{n,i}$. Letting $p_{ij} : \mathcal M \times X^{m_i} \to \mathcal M \times X$ be the projection to $\mathcal M$ times the $j$-th copy of $X$, we can consider the intersection \[\bigcap_j p_{ij}^* h^* Z_{n,i} \subset \mathcal M \times X^{m_i}\] over which we have the bundle \[\bigotimes_j p_{ij}^* \mathbb V_i,\] which descends to the quotient by the symmetric group $\bigcap_j p_j^* h^* Z_{n,i}/S_{m_i} \subset \mathcal M \times X^{[m_i]}$. Finally, for $i \ge 1$ consider the projection $q_i$ from $\mathcal M \times J^l(X) \times X^{[m_1]} \times \cdots \times X^{[m_{n-1}]}$ to $\mathcal M \times X^{[m_i]}$ and $q_0$ the projection to $\mathcal M \times J^l(X)$. Over the intersection 
\begin{equation}\label{big_intersect}
	\bigcap_{i>0} q_i^{-1} \bigcap_j p_{ij}^* h^* Z_{n,i}
\end{equation}
we have the bundle
\begin{equation} \label{tilde_lambda}
	\widetilde \Lambda^0 \coloneqq q_0^* \Lambda^{(0)} \otimes \bigotimes_i q_i^* \bigotimes_j p_{ij}^* \mathbb V_i.
\end{equation}

We obtain $\widetilde{\mathcal M}$ by further intersecting \eqref{big_intersect} with $\mathcal M \times F_\alpha^{\mathrm{vs}}$, where by $F_\alpha^{\mathrm{vs}}$ we mean the very stable locus of $F_\alpha$.
 satisfies $(\widetilde \Lambda^0)_{ ((E,\Phi),\mathcal E)} \cong (\Lambda^0_{\mathcal E}(h(E,\Phi)))_{(E,\Phi)}$ as in \eqref{lambda_zero}.


 Finally, using the projection $p:\widetilde{\mathcal M} \to \mathcal M^0$, we can consider the pushforward $p_* \widetilde \Lambda^0$ over $\mathcal M^0$. Notice that, if $h(E,\Phi) = a \in \mathcal A^0$, then
\[(p_* \widetilde \Lambda^0)_{(E,\Phi)} = \bigoplus_{\mathcal E \in F_\alpha(a)} (\Lambda^0_{\mathcal E}(a))_{(E,\Phi)},\] i.e., over each fiber $h^{-1}(a)$, the pushforward $\Lambda^0_\alpha \coloneqq p_* \widetilde \Lambda^0$ agrees with \eqref{lambda_zero_ha}.
\end{proof}

One can also see the bundle $\widetilde \Lambda^0$ in \eqref{tilde_lambda} as a subbundle of $\widetilde \Lambda$ using the following algebraic lemma.

\begin{lemma}
	Let $(A_i : V_i \to V_i)_{i=1,\ldots,N}$ be a finite collection of endomorphisms. There exists a canonically defined linear map with domain $\bigotimes_i V_i$ whose kernel is precisely $\bigotimes_i \ker A_i$. 
\end{lemma}

\begin{proof}
	For each $k$, consider the endomorphism of $\bigotimes_i V_i$,
	\[\widetilde{A_k}\coloneqq \id_{1} \otimes \cdots \otimes A_k \otimes \cdots \otimes \id_N, \] where $\id_k \coloneqq \id_{V_k}$. Its kernel is 
	$V_1 \otimes \cdots \otimes \ker A_k \otimes \cdots \otimes V_N$.  Since
	$\bigcap_i \ker \widetilde{A_i} = \bigotimes_i \ker A_i$, the linear map
	\[\bigoplus_i \widetilde A_i : \bigotimes_i V_i \longrightarrow \bigoplus_k \bigotimes_i V_i \] has the desired property.
\end{proof}

In this way, by pulling back and tensoring all the maps $\Lambda^{n-i}\mathbf \Phi : \Lambda^{n-i}\mathbb E \to \Lambda^{n-i}\mathbb E$ to $\mathcal M \times F_\alpha$ following the same procedure as in the construction of the bundle $\Lambda$, we can realize the bundle $\widetilde \Lambda^0$ as a subbundle of $\widetilde \Lambda$ over $\widetilde {\mathcal M}$.

Thus we have shown the following analogue of Theorem \ref{HH_mirror} for the Lagrangian $\mathcal L_\alpha$.

\begin{prop}\label{mirror_L}
		The relative Fourier-Mukai transform $S$ over $\mathcal \{0\} \times A^0$ satisfies
	\[S\left(\mathcal O_{\mathcal L_\alpha \cap \mathcal M^0}\right) = \Lambda^0_\alpha.\]
\end{prop}

To pass from a bundle on $\mathcal M(\GL)$ to one on $\mathcal M(\PGL)$ -- the mirror of $\mathcal M = \mathcal M(\SL)$ -- we check that it descends to the quotient at each fiber of the Hitchin map. For the group $\PGL(n,\C)$, the fiber of the Hitchin map over $a \in \mathcal A$ can be identified with $\Prym(X_a)^\vee$, the dual of the Prym variety associated to the spectral cover $\pi:X_a \to X$. Let $i:\Prym(X_a) \to \Jac(X_a)$ be the inclusion. Then from \cite[Equation (11.3.3)]{Polishchuk_2003}, we have
\[S_J \circ i_* = \hat i^* \circ S_P,\] where $S_J$ is the Fourier-Mukai associated to the Jacobian, $S_P$ the one associated to the Prym and $\hat i:\Jac^\vee \to \Prym^\vee$ is the dual map. Thus, if $M \in \Prym(X_a)$, the Fourier-Mukai of the skyscraper sheaf on $\Jac(X_a)$ supported on $M$ is the pullback of a sheaf on $\Prym^\vee(X_a)$. In other words, we have the following Cartesian diagram
\[
\begin{tikzcd}
	S_J( i_*\mathcal O_M) \arrow[r] \arrow[d] & S_P(\mathcal O_M) \arrow[d]\\ 
	\Jac^\vee(X_a) \arrow[r,"\hat i"] & \Prym^\vee(X_a).
\end{tikzcd}
\]

Thus, as a corollary of Proposition \ref{mirror_L}, we have determined a vector bundle over an open dense set of $\mathcal M(\PGL)$ that over each fiber of the Hitchin map $h^{-1}(a)$, agrees with the Fourier-Mukai transform of the structure sheaf of the intersection $\mathcal L_\alpha \cap h^{-1}(a)$.


\begin{theorem}
	\label{thm:D}
	\ThmD
\end{theorem}

In order to complete the program of \cite{HH} applied to the Lagrangian $\mathcal L_\alpha$, we have to conclude that the vector bundle $\Lambda^0_\alpha$ we produced is indeed its mirror BBB brane. It remains to check that it is hyperholomorphic, i.e., it carries a connection of type $(1,1)$ with respect to the three complex structures of the moduli space $\mathcal M$. This is a problem that we hope to come back to in future work. 

\newpage
\clearpage
\phantomsection
\printbibliography
\end{document}